\documentclass[a4paper]{article}

\usepackage[centertags]{amsmath}
\usepackage{amscd}
\usepackage{amsthm}
\usepackage{amssymb}
\usepackage{graphicx}
\usepackage[utf8]{inputenc}
\usepackage{tikz}
\usepackage{verbatim}
\usepackage{booktabs}
\usepackage{framed}

\theoremstyle{definition}
\newtheorem{Def}{Definition}[section]
\theoremstyle{plain}
\newtheorem{Lem}[Def]{Lemma}
\newtheorem{Cor}[Def]{Corollary}
\newtheorem{Pro}[Def]{Proposition}
\newtheorem{Teo}[Def]{Theorem}
\newtheorem{Con}[Def]{Conjecture}
\newtheorem{TI}{Theorem}

\newtheorem{CI}[TI]{Conjecture}

\theoremstyle{remark}
\newtheorem{Rem}[Def]{Remark}
\newtheorem{Exa}[Def]{Example}

\makeatletter
\@addtoreset{table}{section}
\makeatother

\newcommand{\ord}{\operatorname{ord}}

\newcommand{\Spec}{\operatorname{Spec}}
\newcommand{\Pic}{\operatorname{Pic}}

\newcommand{\vol}{\operatorname{vol}}
\newcommand{\mult}{\operatorname{mult}}
\newcommand{\cent}{\operatorname{center}}
\newcommand{\Proj}{\operatorname{Proj }}
\newcommand{\C}{{\mathbb C }}
\newcommand{\R}{{\mathbb R }}

\newcommand{\Z}{\mathbb{Z}}
\newcommand{\Q}{\mathbb{Q}}
\newcommand{\N}{\mathbb{N}}

\newcommand{\A}{\mathbb{A}}
\renewcommand{\P}{\mathbb{P}}
\newcommand{\I}{{\cal I}}

\newcommand{\T}{{\cal T}}
\renewcommand{\H}{{\cal H}}
\renewcommand{\O}{{\cal O}}

\newcommand{\NE}{ {\rm NE}}
\newcommand{\Mor}{ \overline{\rm NE}}
\newcommand{\Nef}{{\rm Nef}}
\renewcommand\hat{\widehat}

\makeatletter
\def\vv{\@ifnextchar[{\@withv}{\@withoutv}}
\def\@withv[#1]#2#3{v(#2,#3;#1)}
\def\@withoutv#1#2{v(#1,#2)}
\makeatother

\makeatletter
\def\mm{\@ifnextchar[{\@withm}{\@withoutm}}
\def\@withm[#1]#2#3{\mu_{#1}(#2,#3)}
\def\@withoutm#1#2{\hat\mu(#1,#2)}
\makeatother

\def\keywordname{{\bfseries Keywords}}%
\def\keywords#1{\par\addvspace\medskipamount{\rightskip=0pt plus1cm
\def\and{\ifhmode\unskip\nobreak\fi\ $\cdot$
}\noindent\keywordname\enspace\ignorespaces#1\par}}
\def\subclassname{{\bfseries Mathematics Subject Classification
(2000)}\enspace}
\def\subclass#1{\par\addvspace\medskipamount{\rightskip=0pt plus1cm
\def\and{\ifhmode\unskip\nobreak\fi\ $\cdot$
}\noindent\subclassname\ignorespaces#1\par}}

\begin{document}

\author{Marcin Dumnicki, Brian Harbourne,
Alex K\"uronya,
\\ Joaquim Ro\'e, and Tomasz Szemberg}
\title{Very general monomial valuations of $\P^2$
and a Nagata type conjecture}

\date{}

\maketitle

\begin{abstract}
It is well known that multi-point Seshadri constants
for a small number $t$ of points in the projective plane are submaximal.
It is predicted by the Nagata conjecture that their values are maximal
for $t\geq 9$ points. Tackling the problem in the language of valuations
one can make sense of $t$ points for any real $t\geq 1$.
We show somewhat surprisingly that a Nagata-type conjecture
should be valid for $t\geq 8+1/36$ points and we compute
explicitly all Seshadri constants
(expressed here as the asymptotic maximal vanishing element)
for $t\leq 7+1/9$. In the range $7+1/9\leq t\leq 8+1/36$
we are able to compute some sporadic values.
\keywords{Nagata Conjecture, SHGH Conjecture, Seshadri constants, monomial valuations, anticanonical divisor}
\subclass{MSC 14C20, 13A18}
\end{abstract}

\section{Introduction}
The main purpose of this work is to formulate an
analogue of Nagata's conjecture which makes sense for real values
$t\ge 1$ of the number of points blown up instead of
integral ones.
Using quasi-monomial valuations of the plane
we construct a function $\hat\mu:[1,\infty) \rightarrow \R$,
such that if $\hat\mu(t)=\sqrt{t}$
for all integers $t\ge 9$ then Nagata's conjecture is true.
Moreover we show that $\hat\mu$ is a continuous function and we
propose a conjecture that asserts the equality $\hat\mu(t)=\sqrt{t}$
for $t\geq 8+1/36$ (Conjecture \ref{mainconj}). This
fits well with the expected behavior of linear systems on blow-ups
of $\P^2$, as it would follow from a stronger open conjecture
by G.~M.~Greuel, C.~Lossen and E.~Shustin. On the other hand, the
behavior of the function $\hat\mu$ in the range $7+1/9\leq t\leq 8+1/36$
is somewhat mysterious.

By continuity, it suffices to verify the conjecture at  \emph{rational square}
values of $t$, which boils down to verifying nefness of appropriate
divisor classes with selfintersection zero. Thus Nagata's conjecture
is reduced to proving a statement of a kind which has shown to be tractable,
see \cite{CHMR}, \cite{Bir}.
These selfintersection zero classes live on blow-up configurations
that have not been known earlier to shed light on the original conjecture
for ten or more points.

Going further down this road we
discuss the value of $\hat\mu$ at many non-integral cases and compute it
in a wide range including all $t\le 7+1/9$. As a tool and also as a result
of independent interest, we describe the Mori cone of the related blown up
surfaces whenever they are anticanonical.

\medskip

In our approach valuations are considered
as a generalization of points,  a natural step taken
in many situations ever since Zariski's pioneering work.
In the context of
linear systems defined by multiple base points on projective varieties,
positivity, and Seshadri constants,
it is a point of view which seems to have been explored explicitly
only recently.
In \cite{DKMS} and \cite{BKMS}, S. Boucksom, M. Dumnicki,
A. K\"uronya, C. Maclean, and T. Szemberg introduced the constant
$a_{\max}$ of a valuation (here denoted $\hat\mu$),
analogous to the $s$-invariant introduced by L.~Ein, S.~D.~Cutkosky and R.~Lazarsfeld
in \cite{CEL01} for ideals (see also \cite[5.4]{PAG}). For a valuation $v$
centered at the origin of $\A^2=\Spec \C[x,y]$, one has by definition
$$\hat\mu(v)=\lim_{d\to \infty}\frac{\max\{v(f)\,|\,f\in \C[x,y],{\deg f\le d}\}}{d}.$$
All such invariants encode essentially the same information
as the Seshadri constant does in the case of points and, as is the case for
Seshadri constants, they turn out to be extremely hard to compute.

The last decade has also seen the blossoming of a geometric study of
spaces of \emph{real} valuations (see C.~Favre--M.~Jonsson \cite{FJ04})
or spaces of \emph{seminorms},
usually called Berkovich spaces \cite{Ber90}, which essentially coincide in dimension two
(see M.~Jonsson \cite[section 6]{Jon} for a description in the plane case).
Being compact and arcwise connected,  the topology of such spaces has very interesting and useful
properties.
The work of S.~Boucksom, C.~Favre and M.~Jonsson \cite{BFJ09}, \cite{BFJ}
implicitly reveals connections between such valuation spaces,
positivity, and birational geometry.

\vskip\baselineskip
In this paper  the invariant $\hat\mu$ is studied as a function on
the space $\mathcal V$ of plane valuations of real rank 1.
This invariant turns out to be lower semicontinuous and continuous along arcs in $\mathcal V$
(Theorem \ref{lower-semicontinuous}).
There is no
difficulty in extending the definition of $\hat\mu$ to other
varieties; one obtains a function-invariant for line bundles
whose geometric significance would deserve further study.
%
Motivated by what is known in  the case of points and by the
conjectures of Nagata and Segre--Harbourne--Gimigliano--Hirschowitz,
our  focus will be on valuations along a \emph{very general}
half-line  in $\mathcal V$.

We let $\hat\mu(t)=\hat\mu(v_t)$ with $t\in[1,\infty)$, where $v_t$ is
a very general quasimonomial valuation with characteristic exponent $t$
 (see Section \ref{sec:preliminaries}  
 for precise definitions).

Our main results, Theorems \ref{polyhedral} and
\ref{lastthm}, are the first steps toward the computation of $\hat \mu$.
Divisorial valuations are dense in each arc of the valuation space;
our tools provide a good grip on such valuations, and we work on  the
minimal proper birational model $X_t$ where the center
of $v_t$ is a divisor.
When $X_t$ supports an effective anticanonical divisor,
extensive knowledge of its geometry is available,
see \cite{Har97}, \cite{Harb97}.
In section \ref{sec:anticanonical} we determine the range of $t$
for which $X_t$ is anticanonical, and study the Mori cone of $X_t$
in that range. The following theorem sums up the main results
of section \ref{sec:anticanonical}.
\begin{TI}
Let $v_t$ be a very general quasimonomial valuation on $\P^2$ with
characteristic exponent $t\in\Q$, and let $X_t$ be the minimal model
where $v_t$ has divisorial center. $X_t$ supports an effective
anticanonical divisor if and only if
$1\le t\le 7, t=7+1/n$ for some natural number $n$, or $t=9$.

If $1\le t\le 7$, then the Mori cone $\Mor(X_t)$ is a polyhedral cone,
spanned by the classes of the exceptional
components of $X_t\rightarrow \P^2$,
the class of a particular nodal cubic,
and finitely many $(-1)$-curves (whose number is explicitly bounded,
see \ref{polyhedral}).

If $t=7+1/n$ for natural $n$, then the only prime divisors $C$ in $X_t$
with $C^2\le -2$ are exceptional components, and
$\Mor(X_t)$ is a polyhedral cone
if and only if $n\le 8$.

If $1\le t\le 3$, $t=3+1/n$ for natural $n$, or $t=5$, then the monoid
of effective classes can be generated by the classes of the
components of the exceptional divisor, a particular conic,
and the $(-1)$-curves.
\end{TI}

It is not hard to see that $\hat\mu(t)\ge \sqrt t$, and one
should expect the equality to hold unless
there is a good geometric reason, in the form of
a $(-1)$-curve $C_t$ on $X_t$ with value higher  than $\deg C_s\cdot \sqrt t$ .

 \begin{CI}\label{mainconjintro}
 For every $t\ge 8+1/36$, $\hat\mu(t)=\sqrt{t}$.
 \end{CI}

In section \ref{sec:Nagata} we explore the relations of conjecture
\ref{mainconjintro} and existing conjectures,
showing in particular that Nagata's conjecture
is just a special case of conjecture \ref{mainconjintro}.
If $t$ is an integer, then it is the number of points that have been blown up
to construct $X_t$, and we look at  $\hat\mu(t)$
as a continuous function that interpolates between the
inverses of Seshadri constants at $t$ very general points, whose values
at non-integer $t$ also have geometric meaning.
In addition,  it is not hard to show (Proposition~\ref{squares})
that for  integer values of $t$ that are squares,  $\hat\mu(t)= \sqrt t$ holds.
A further, stronger conjecture, motivated by our main results,
is proposed at the end of section \ref{sec:supraminimal}.

Knowing the cone of curves allows to compute $\hat\mu$, which for
small $t$ is done in section~\ref{sec:supraminimal}. Denote $F_{-1}=1$, $F_0=0$ and $F_{i+1}=F_i+F_{i-1}$ the Fibonacci numbers,
and $\phi=(1+\sqrt{5})/2=\lim F_{i+1}/F_i$ the ``golden ratio''.

\begin{TI} \label{muvalues}
The value of $\hat\mu(t)$ for $t\in [1,\phi^4]$ is given by
  $$\hat\mu(t)=
\begin{cases}
 \frac{F_{i-2}}{F_{i}}\,t & \text{if }t \in \left[\frac{F_{i}^2}{F_{i-2}^2},\frac{F_{i+2}}{F_{i-2}}\right] \, ,\\
 \frac{F_{i+2}}{F_{i}} & \text{if }t \in \left[\frac{F_{i+2}}{F_{i-2}},\frac{F_{i+2}^2}{F_{i}^2}\right] \, ,
\end{cases}
$$
where $i\ge 1$ takes all odd values. For $t\in [\phi^4,7+1/9]$,
  $$\hat\mu(t)=
\begin{cases}
 \frac{1+t}{3} & \text{if }t \in \left[\phi^4,7\right]\, ,\\
 \frac{8}{3} & \text{if }t \in \left[7,7+1/9\right]\, .
\end{cases}
$$
\end{TI}
In particular there is a sequence of \emph{rational} squares $t<8$ with
$\hat\mu(t)= \sqrt t$, with an accumulation point at $\phi^4$;
we suspect that at least some rational squares
$t>9$ can be dealt with by existing techniques, which by continuity of
$\hat \mu$ would allow to compute $\hat\mu(t)$ for nonsquare $t$.

For anticanonical $X_t$
there exists a $(-1)$-curve computing $\hat\mu(t)$.
This implies that $\hat\mu$ is piecewise linear near $t$.
We describe a (countably infinite) family of $(-1)$-curves
from which Theorem \ref{muvalues} follows, and also
determine $\hat\mu(t)$ for other small
values of $t$ (see Figure \ref{graph}).
We conjecture that this list is complete.
If that is indeed so, then in particular
$\hat\mu(t)=\sqrt t$ for $t\ge 8+1/36$.
Except for
9 cases, the  $(-1)$-curves of Section~\ref{sec:supraminimal} are
the same unicuspidal curves which are known to give the asymptotically
extremal ratio between degree and multiplicity,
as explained in  Y.~Orevkov's work \cite{Ore02}
(see also the overview \cite{FLMN07}).

\begin{figure}
 \includegraphics[width=\linewidth,keepaspectratio=true]{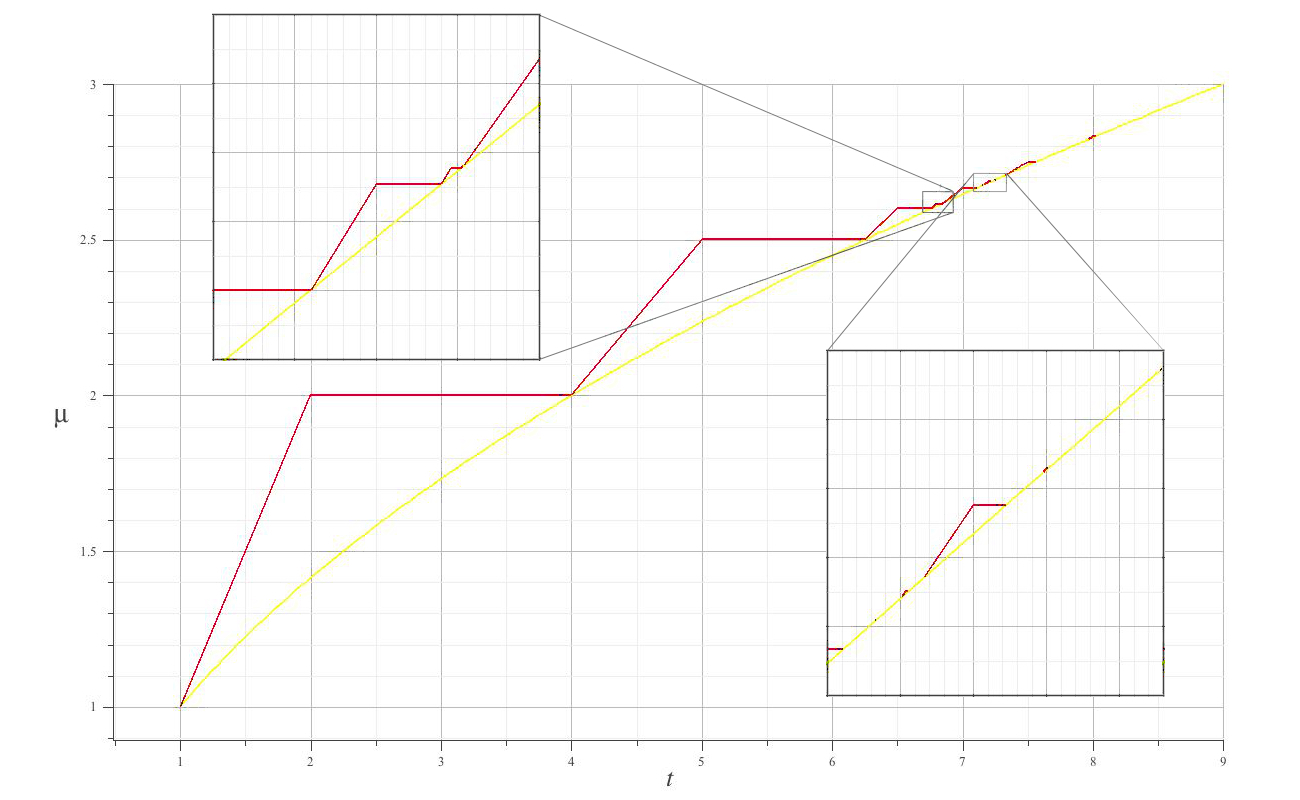}
 \caption{In red, the known behavior of $\hat\mu(t)$ for $t\le 9$; in yellow, the lower bound
 $\sqrt{t}$. \label{graph}}
\end{figure}

In what follows we work over the field of complex numbers.

\section{Preliminaries}
\label{sec:preliminaries}

We refer to the references
O.~Zariski--P.~Samuel \cite[Chapter VI. and Appendix 5.]{ZS75} and
E.~Casas--Alvero \cite[Chapter 8]{Cas00}
for the general theory of valuations and complete ideals on surfaces.
Let us now briefly recall the definitions and facts needed
for the definition of $\hat\mu$ and the statement of the conjecture.

Let $v$ be a rank 1 valuation
(meaning that the value group is an ordered subgroup of $\mathbb{R}$)
on the field of functions $F$ of a projective algebraic surface $S$.
For every effective divisor $D\subset S$, denote
$v(D)$ the value of any equation of $D\cap U$,
where $U$ is an affine chart intersecting $D$.

Following \cite{BKMS},
we denote
\[
\mu_D (v) = \max\{v(D')\,|\, D'\in |D| \}\ , \quad
\text{and} \quad
\hat \mu_D (v) = \lim_{k \to \infty}\frac{\mu_{kD}(v)}{k}\ .
\]
For every non-negative $m\in \R$, the ideal sheaves
\[
\I_m=\{f\in \O_S\,|\, v(f)\ge m \}, \quad \text{and} \quad
\I_m^+=\{f\in \O_S\,|\, v(f)> m \}
\]
are called valuation ideals.
The closed subscheme
defined by $\I_0^+$ is an irreducible subvariety,
called \emph{center} of the valuation, $\cent(v)$.
If $R_v$ denotes the valuation ring of $v$, the generic point of the center
is the image of the closed point under the unique
map $\Spec R_v \rightarrow S$ that exists by the valuative
criterion of properness. Of course, all this continues to apply
if we substitute $S$ by another projective model $S'$
(i.e., a smooth projective surface
with a fixed isomorphism $K(S')\cong F$).

If the $\cent(v)$ is a curve $C$, then $v$ is
(up to a constant $c\in \R$)  the order of vanishing
along $C$; thus, $v(D)=c \cdot \ord_C D=c \cdot \max\{k\,|\,D-kC\ge 0\}$.

We are mostly interested in valuations of $S=\P^2$ such that the $\cent(v)$
is a closed point. In this case the volume of $v$, as defined in \cite{ELS05}, is
$$\vol(v) \,:=\,\lim_{m\to\infty}
\frac{\dim_{\C} (\O_S/\mathcal{I}_m)}{m^2/2}$$
(note that $\O_S/\mathcal{I}_m$ is an artinian
$\C$-algebra supported
at the center of the valuation)
and the volume of a divisor class $D$ on a surface $S$
of dimension $d$ is defined as
\[
\vol (D) := \limsup_{k\to\infty} \frac{h^0(S, kD)}{k^2/2} .
\]
Boucksom-K\"uronya-MacLean-Szemberg show that the invariant $\hat\mu$
can be bounded in terms of values in arbitrary dimension; let us recall
their result in the case of surfaces:
\begin{Pro}[{\cite[Proposition 2.9]{BKMS}}]\label{volBKMS}
Let $D$ be a big divisor and $v$ a real valuation centered at a point
$p\in S$. Then
\begin{equation*}
 \hat \mu_D (v)\ge \sqrt{\vol(D)/\vol(v)}\ .
\end{equation*}
\end{Pro}
When $D$ is ample this is equivalent to the bound
\(
 \hat\mu_D(v)\ge\sqrt{D^2/\vol(v)} \ .
\)
Valuations which satisfy the equality in Proposition \ref{volBKMS}, with
$D\subset S=\P^2$ a line, will be called
\emph{minimal}.

For the sake of simplicity we recall the notion of quasimonomial valuations
specializing to the case when  $S=\P^2$ and the center of $v$
is the origin $(0,0)\in\A^2=\Spec \C[x,y] \subset \P^2=\Proj \C[X,Y,Z]$,
with $x=X/Z, y=Y/Z$.
In this situation we write
\[
\mu_d (v) = \max\{v(f)|f\in \C[x,y], \deg f \le d\}\ , \quad
\text{and} \quad
\hat \mu(v) = \lim_{d\to \infty}\frac{\mu_d(v)}{d}\ .
\]

\begin{Def}\label{seriesdef}
 Given a series $\xi(x)\in \C[[x]]$ with $\xi(0)=0$ and a real number $t\ge 1$,
let
\[
\vv[f]{\xi}{t} := \ord_x(f(x,\xi(x)+\theta x^t))\ ,
\]
where the symbol $\theta$ is transcendental over $\C$.

Equivalently, expand $f$ as a Laurent series
\[
f(x,y) = \sum a_{ij}x^i (y-\xi(x))^j \ ,
\]
and put
\begin{equation}\label{monomialdef}
 \vv[f]{\xi}{t} := \min\{i+tj|a_{ij\ne 0}\}\ .
\end{equation}
Then $f\mapsto \vv[f]{\xi}{t}$ is a valuation which we denote $\vv{\xi}{t}$.
Such valuations are called \emph{monomial} if
$\xi=0$, and \emph{quasimonomial} in general.
Slightly abusing language, $t$ will be called the
\emph{characteristic exponent} of $\vv{\xi}{t}$ (even if it is an integer).
For simplicity we also write
\[
\mu_d(\xi,t) = \mu_d(\vv{\xi}{t})\ , \quad
\text{and} \quad
\hat\mu(\xi,t)=\hat \mu (\vv{\xi}{t})\ .
\]
\end{Def}

\begin{Rem}
The valuation $\vv{\xi}{t}$ depends only on the
$\lfloor t \rfloor$-th jet of $\xi$, so for fixed $t$ this series can be safely assumed
to be a polynomial; however, later on we'll let $t$ vary for a fixed $\xi$.
\end{Rem}

It is not difficult to see directly using \eqref{monomialdef}, and will
be proved using geometric considerations in the next subsection 
that $\vol(\vv{\xi}{t})=t^{-1}$.
The precise statement of Conjecture \ref{mainconjintro}
is now:

\begin{Con}\label{mainconj}
 For a sufficiently general choice of $\xi$, and every $t\ge 8+1/36$,
 the valuation $\vv{\xi}{t}$ is minimal.
\end{Con}

\subsection*{Cluster of centers of a valuation}

Next we introduce the geometric structures attached to valuations
$\vv{\xi}{t}$ which allow us to study $\hat \mu (\xi,t)$ and justify
the conjecture.

Each valuation with 0-dimensional center naturally determines a
\emph{cluster of centers}, as follows.
To begin with, let $p_1=\cent(v)$ in the projective surface $S$.
Consider the blowup $\pi_1:S_1\rightarrow S$
centered at $p_1$ and let $E_1$ be the corresponding exceptional divisor.
The center of $v$ on $S_1$ may be  $E_1$ or a point $p_2\in E_1$.

Iteratively blowing up the centers $p_1, p_2, \dots$ of $v$ either ends
with a model where the center of $v$ is an exceptional divisor $E_n$,
in which case
\[
v(f) \,=\, c\cdot \ord_{E_n}f
\]
for some constant $c$, and  $v$ is called a divisorial valuation,
or this process goes on indefinitely.
For each center $p_i$ of $v$, general curves through $p_i$
and smooth at $p_i$
have the same value $v_i=v(E_i)$.

Following \cite[Chapter 4]{Cas00}, we call the sequence
$K=(p_1, p_2,\dots)$, with weights $v_i=v(E_i)$,
a weighted 
cluster of points, which completely determines $v$.
Indeed, for every effective divisor $D \subset S$,
\begin{equation}
\label{eq:proximity}
 v(D)=\sum_i v_i\cdot \mult_{p_i} \widetilde D_i,
\end{equation}
 where $\widetilde D_i$ denotes
proper transform at $S_i$.  The sum may be infinite, but for  valuations
with real rank 1, which are the ones we consider here,
$\widetilde D$ can have positive multiplicity at only a finite number of centers
\cite[8.2]{Cas00}.

Sometimes we shall say that a divisor goes through
an infinitely near point to mean that its proper transform
on the appropriate surface goes through it.

\begin{Def}
With notation as above, given indices $j<i$, the center $p_i$ is  called \emph{proximate to $p_j$}
($p_i\succ p_j$) if $p_i$ belongs to the proper transform $\widetilde E_{j}$ of the exceptional
divisor of $p_j$. Each $p_i$ with $i>0$ is proximate to $p_{i-1}$ and to
at most one other center $p_j$, $j<i-1$; in this case
$p_i=\widetilde E_{{j}}\cap E_{i-1}$ 
and $p_i$ is called a   \emph{satellite point}. A point which is not a satellite point is called \emph{free}.
\end{Def}

\begin{Rem}
\label{proximate-irrdiv}
  The irreducible components of exceptional divisors
can be computed as proper transforms if the proximity
relations are known: $\tilde E_j=E_j-\sum_{p_i\succ p_j}E_i$
\end{Rem}

\begin{Rem}\label{rm:proximity}
For every valuation $v$, and every center $p_i$ such that $v$ is not the
divisorial valuation associated to $p_i$, equation~\eqref{eq:proximity}
applied to $D=E_{j}$ gives rise to the so-called  \emph{proximity equality}
\[
v_j \,=\, \sum_{p_i \succ p_j} v_i \ .
\]

For effective divisors $D$ on $S$, the intersection number
$\widetilde D \cdot \widetilde E_{j}\ge0$
together with remark \ref{proximate-irrdiv}
yield the proximity inequality
\[
\mult_{p_j}(\widetilde D_j)\ge \sum_{p_i\succ p_j} \mult_{p_i} (\widetilde D_i)\ .
\]
\end{Rem}
\vskip\baselineskip

Assume now that  $v=\ord_{E_s}$ is the divisorial valuation with cluster of centers $K=(p_1,\dots,p_s)$,
while $\pi_K:S_K\rightarrow S$ denotes the composition of the blowups of all points of $K$.
Then,  for every $m>0$,  the valuation ideal sheaf $\I_m$ can be described as
\[
\mathcal I_m = (\pi_K)_* (\O_{S_K}(-mE_s))\ .
\]

\begin{Rem}\label{UnloadRem}
As soon as $s>1$, the negative intersection number
$-mE_s\cdot \widetilde E_{s-1}=-m$ implies that all global sections of
$\O_{S_K}(-mE_s)$ vanish along $\widetilde E_{s-1}$,  and therefore
\[
\mathcal I_m = (\pi_K)_* (\O_{S_K}(-mE_s-\widetilde E_{s-1})) =
(\pi_K)_* (\O_{S_K}(-E_{s-1}-(m-1)E_s))\ .
\]
This \emph{unloads} a unit of multiplicity from $p_s$ to $p_{s-1}$.
The finite process of subtracting
all exceptional components that are met negatively, (i.e.,
starting from a divisor $D_0=-m_1E_1-\cdots-m_sE_s$
and successively replacing $D_i$ by $D_i-\widetilde E_j$, starting with $i=0$,
whenever $D_i\cdot \widetilde E_j<0$ for some $j$,
until one obtains a $D_i$ such that
$D_i\cdot \widetilde E_j\geq 0$ for all $j$) is
classically called \emph{unloading the weights of the cluster}.
The final uniquely determined system of weights $\bar m_i$
satisfies
\[
D_{m} = -\sum \bar m_i E_i \ \ \text{ is nef relative to }\pi_K
\]
(recall that a divisor is nef relative to a morphism $f$
when it intersects nonnegatively every curve mapping to a point
\cite[1.7.11]{PAG})
and
\[
\mathcal I_m = (\pi_K)_* (\O_{S_K}(D_m))\ .
\]

In this case, general sections of $\mathcal I_m$
have multiplicity exactly $\bar m_i$ at $p_i$, and  no other
singularity. More precisely,  for any ample divisor
class $A$ on $S$, the complete system $|kA+D_m|$
 for $k\gg0$ is base-point-free (were we denote $A=(\pi_K)^*(A)$)
and its general elements are smooth and meet
each $E_j$ transversely at $\bar m_j- \sum_{p_i \succ p_j} \bar m_i$
distinct points.
Note that relative nefness of $D_m$ is equivalent to the proximity
inequality $\bar m_j\ge \sum_{p_i \succ p_j} \bar m_i$.

It follows using \eqref{eq:proximity} that the valuation of an effective
divisor $D$ on $S$ can be computed as a local intersection multiplicity
\[
v(D) = I_{p_1}(D,C)
\]
where $C$ is the image in $S$ of a general element of $|kA+D_m|$.
\end{Rem}
\vskip\baselineskip

The unloading procedure just described also yields the following.

\begin{Lem}\label{simple}
 Let $v=\ord_{E_s}$ be the divisorial valuation whose cluster of centers is
 $K=(p_1,\dots,p_s)$ with weights $v_i$, and for every $m>0$ denote
 $D_m=-\sum \bar m_i E_i$  the unique nef divisor relative to  $\pi_K$ with
 $\mathcal I_m=(\pi_K)_* (\O_{S_K}(D_m))$. If  $m=k\sum v_i^2$ for
 some integer $k$,  then $\bar m_i=kv_i$ for all $i$.
\end{Lem}

\begin{proof}
It is clear that  $-\sum \bar m_i E_i$ is nef relative to $\pi_K$ because
of the proximity equalities from remark \ref{rm:proximity}. Moreover, because
every effective divisor $D$ satisfies the proximity inequalities,
if $\mult_{p_1}D<\bar m_1$ then $\mult_{p_i}D<\bar m_i$ for all $i$, and by
equation \eqref{eq:proximity}, $v(D)<m$. Arguing by induction on $s$,
one sees that  $\mathcal I_m=(\pi_K)_* (-\sum \bar m_i E_i)$.
\end{proof}

\begin{Rem} Write $m_0$ for $\sum v_i^2$. Then,
in the context of Zariski's theory of factorizations of complete ideals,
lemma \ref{simple} translates into
\[
\mathcal I_{km_0}=\mathcal I_{m_0}^k\ ,
\]
and to the fact that $\mathcal I_{m_0}$  is a simple complete ideal.
For other values of $m$  one has instead
\[
\mathcal I_{km_0+\delta}=\mathcal I_{m_0}^k\mathcal I_\delta\, .
\]

Non-divisorial valuations can be considered to be limits of divisorial valuations
and their valuation ideals turn out to be  complete as well, determined by finitely
many centers. The ideal $\mathcal I_{km}$ is then never a power of $\mathcal I_m$,
rather  there exists  $\delta>0$ such that
\[
\mathcal I_m^k\,\subset\, \mathcal I_{km}\,\subset\, \mathcal I_{m-\delta}^k
\]
for all $m$ and $k$. Such bounds actually hold in greater generality, namely for Abhyankar valuations in arbitrary dimension; see \cite{ELS05}
by L.~Ein, R.~Lazarsfeld and K.~Smith.
\end{Rem}
\vskip\baselineskip

\begin{Lem}\label{volume_from_cluster}
  Let $v=\ord_{E_n}$ be the divisorial valuation with  cluster of centers
 $K=(p_1,\dots,p_n)$ and  weights $v_i$. Then
\[
\vol (v)\,=\,\Big(\sum v_i^2\Big)^{-1}\ .
\]
\end{Lem}

\begin{proof}
 For $m=k\sum v_i^2$,
 $\dim_\C (\O_X/\mathcal{I}_m)=\sum kv_i(kv_i+1)/2$
by \cite[4.7]{Cas00}.
\end{proof}

\begin{Rem}
 It is proven by Cutkosky and Srinivas in \cite[Corollary 1]{CutSri} that divisorial valuations on surfaces have rational volume under mild
conditions. On the other hand \cite[Theorem 1.1]{Kur03} shows that this is not the case in higher dimensions.
\end{Rem}

\vskip\baselineskip
Consider the group of numerical equivalence classes of $\R$-divisors
$N_1(S_K)$, where  $S_K$ is the
blowup at the cluster of centers of $v$.
One calls a  rational ray in $N_1(S_K)$  \emph{effective}, if it is generated by an effective  class.
The \emph{Mori cone} $\Mor(S_K)$ is the closure in $N_1(S_K)$  of the set $\NE(S_K)$ of all effective rays,
and it is the dual of the \emph{nef cone} $\Nef(S_K)$ which is the closed cone described by all nef rays.

A \emph{$(-1)$-ray} in $N_1(S_K)$ is a ray generated by a $(-1)$-\emph{curve},
i.e., a smooth, irreducible, rational curve $C$ with $C^2=-1$
(hence $C\cdot \kappa=-1$, where $\kappa$ denotes the canonical class).
Mori's Cone Theorem says that
\[
\Mor(S_K) = \Mor(S_K)^\succcurlyeq + R_n \ ,
\]
where $\Mor(S_K)^\succcurlyeq$  denotes the
subset of $\Mor(S_K)$ described by rays generated by nonzero classes
$\eta$ such that $\eta\cdot \kappa\ge 0$ with  $\kappa$ being the canonical class,
and
\[
R_n \;= \sum_{\rho\;\text {a}\; (-1)-\text{ray}} \rho\; \subseteq\;  \Mor(S_K)^\preccurlyeq\ .
\]

\begin{Rem}
 In cases when $\Mor(S_K)$ is a polyhedral cone,
 Proposition~\ref{volBKMS} yields that $\hat \mu_D (v)$ is
 a rational number, and therefore $v$ can be
 minimal only if $\sqrt{D^2/\vol(v)}$ is rational.
 In fact, all examples of divisorial minimal valuations included here
 correspond to rational values of $\sqrt{D^2/\vol(v)}$,
 even for nonpolyhedral $\Mor(S_K)$.
For some examples of non-divisorial minimal valuations,
see Remark \ref{QuadField}; for these,
$\vol(v)$ defines a quadratic extension of $\mathbb Q$
in which it is a square (i.e., $\sqrt{\vol(v)}\in \mathbb{Q}(\vol(v))$).
\end{Rem}

\subsection*{Centers of a quasimonomial valuation}

 Quasimonomial valuations are exactly the valuations  whose cluster of centers consists
of a few free points followed by satellites, which may be finite or infinite
in number, but not infinitely many proximate to the same center.
We will work with very general quasimonomial valuations on $\P^2$.
The genericity condition refers to the position of the free centers;
it will be made precise below, after describing the  continuity and semicontinuity
properties of $\hat\mu$ on the space of quasimonomial valuations.

\begin{Rem}\label{ContFracEx}\label{ClusterDivRem} \cite{Cas00}
The cluster $K$ of centers of $\vv{\xi}{t}$ can be easily described from
the continued fraction expansion
\[
t=n_1+\frac{1}{n_2+\frac{1}{n_3+\frac{1}{\ddots}}}\ .
\]
 $K$ consists of $s=\sum n_i$ centers; if $t=n_1$ then
they all lie on the proper transform of the germ
\[
\Gamma \colon  y = \xi(x)\  ,
\]
otherwise the first $n_1+1$ lie
on $\Gamma$ and the rest are satellites: starting from $p_{{n_1}+1}$ there
are $n_2+1$ points proximate to $p_{n_1}$, the last of which starts
a sequence of $n_3+1$ points proximate
to $p_{n_1+n_2}$ and so on. If the continued fraction is finite, with
$r$ terms, then the last $n_r$ points (not $n_r+1$) are proximate to
$p_{n_1+\dots+n_{r-1}}$.
The weights are $v_i=1$ for $i=1,\dots,n_1$,
then $v_i=t-n_1$ for $i=n_1+1, \dots, n_1+n_2$, and
$v_i=v_{n_1+\dots+n_{j-1}}-n_j v_{n_1+\dots+n_{j}}$
for $i= n_1+\dots+n_{j}+1, \dots, n_1+\dots+n_{j+1}$.

If $t$ is rational, there are
only finitely many coefficients $n_1,\ldots,n_r$,
so $K=(p_1, p_2,\dots,p_s)$ is finite and the valuation is
divisorial. More precisely,
\[
\vv[f]{\xi}{t} = v_s\cdot \ord_{E_{s}}(f)\ .
\]
The prime divisor components
$\widetilde E_i$ of $E_1$ on $S_K$ can then be described
as follows (where
$\widetilde E_i$, as in Remark \ref{proximate-irrdiv}, is the
proper transform in $S_K$
of the blowup of the point $p_i$).
Note that $s=n_1+\ldots+n_r$;
let $s_i$ be the sum $n_1+\cdots+n_i$, so $s=s_r$.
The only $i$ with $(\widetilde E_i)^2=-1$ is $i=s$,
and in this case $\widetilde E_s=E_s$.
For each $1\leq i<r-1$, we have:
$\widetilde E_{s_i}=E_{s_i}-E_{s_i+1}-\cdots-E_{s_{i+1}+1}$, so
$(\widetilde E_{s_i})^2=-2-n_{i+1}$; for $i=r-1$ we have
$\widetilde E_{s_{r-1}}=E_{s_{r-1}}-E_{s_{r-1}+1}-\cdots-E_{s}$, so
$(\widetilde E_{s_{r-1}})^2=-1-n_r$;
and for every $1\leq j\leq s$ not in the set $\{s_1,\ldots,s_r\}$
we have $\widetilde E_j=E_j-E_{j+1}$, so $(\widetilde E_j)^2=-2$.

If $t$ is irrational, then the sequence of centers is infinite
and the group of values has rational rank 2. There is no surface $S_K$,
but denoting $S_j$ the blowup of the first $j$ points of $K$,
the above description of the divisors $\widetilde E_i$ holds
whenever it makes sense; for instance,
$\widetilde E_{s_i}=E_{s_i}-E_{s_i+1}-\cdots-E_{s_{i+1}+1}$
in every $S_j$ with $j \ge {s_{i+1}+1}$.

\end{Rem}
\vskip\baselineskip

\begin{Cor}
\label{volquasimonomial}
Let $\vv{\xi}{t}$ be a quasimonomial valuation as above. Then
\[
 \vol(\vv{\xi}{t})=t^{-1}\ ,\ \mu_d(\xi,t)\ge d\sqrt t\ ,
\]
and
\[
\hat \mu(\xi,t) \ge \sqrt t \ ,
\]
so 
$\vv{\xi}{t}$ is minimal whenever $\hat \mu(\xi,t)= \sqrt t$.
\end{Cor}
\begin{proof}
The only point that needs proving is the value of $\vol(\vv{\xi}{t})$,
which follows from Lemma \ref{volume_from_cluster}, taking into
account the values $v_i$ computed above and using induction on
the number of terms in the continued fraction of $t$.
\end{proof}
\subsection*{The space of valuations}

 \begin{Rem}\label{tree}
In definition \ref{seriesdef} one may allow formal series
$\xi(x)=\sum_{j\ge 1} a_j x^{\beta_j}$
whose exponents $\beta_j$ form
an arbitrary increasing sequence of rational numbers,
and one still obtains valuations $\vv{\xi}{t}$
(no longer quasimonomial).
It is even possible to allow $t=\infty$, except when
$\xi$ is the (convergent) Puiseux series of a branch of curve
going through the center $p_1$. In this way, \emph{all real valuations}
with center at $p_1$ are obtained (up to a normalizing constant factor,
see \cite[8.2]{Cas00} or \cite[Chapter 4]{FJ04}).


 The most natural topology in the set $\T$ of all real valuations
with center at $p_1$ is the coarsest such that for all $f\in F$,
$v\mapsto v(f)$ is a continuous map $\T \rightarrow \R$.
It is called the \emph{weak topology}. For a fixed
$\xi$, the map $t\mapsto \vv{\xi}{t}$ is then continuous.
There is in $\T$ a finer topology of interest: namely, the finest
topology such that $t\mapsto \vv{\xi}{t}$ is continuous for all $\xi$.
It is called the \emph{strong topology}. With the strong topology,
$\T$ is a profinite $\R$-tree, rooted at the $p_1$-adic valuation
(see \cite{FJ04} for precise definitions and proofs).
To avoid confusion with branches of curve, we call \emph{arcs}
the branches in $\T$.
Maximal arcs
are homeomorphic to the interval $[1,\infty]$ (respectively $[1,\infty)$)
and parameterized by $t\mapsto \vv{\xi}{t}$ where $\xi$
is not (respectively, is) the Puiseux series of a branch of curve
at $p_1$.

The arcs     of $\T$ share the obvious segments given by
coincident jets, and separate at rational values of $t$;
these correspond to divisorial valuations (also in this general case).
 \end{Rem}
\vskip\baselineskip

\begin{Pro}\label{semicont}
 Fix a real number $t>1$ and a natural number $d$. Set $k=\lceil t \rceil$ and denote
 by ${\mathcal J}_{k}\subset \C[[x]]$ the space of $(k-1)$-jets
 of power series with $\xi(0)=0$, endowed with the Zariski topology
 coming from the coefficients map ${\mathcal J}_{k}\cong \A^{k-1}$,
 $\sum a_i x^i\mapsto (a_1,\dots,a_{k-1})$.

 Then  the function
 $\xi\mapsto \mm[d]{\xi}{t}$ descends to
 an upper semicontinuous function
 \[
 {\mathcal J}_{k}\rightarrow \langle 1,t\rangle_\Q \subset \R
 \]
 which takes on only finitely many values.
\end{Pro}

It follows that for fixed $t$, $\mm{\xi}{t}$ takes its smallest
value for $\xi$ with very general jet $\xi_{n-1}$
(i.e., in a countable intersection of Zariski-open subsets of
${\mathcal J}_{k}\cong \A^{k-1}$).

\begin{proof}
Because only the $k$  free centers of $\vv{\xi}{t}$ depend on $\xi$ ($k=n_1$
in the continued fraction expansion if $t$ is an integer
and $k=n_1+1$ otherwise), it is clear that the valuation only depends on the
$(k-1)$-th jet of $\xi$, and the existence of the function
\[
 {\mathcal J}_{k}\rightarrow \langle 1,t\rangle_\Q \subset \R
\]
is clear. We will prove that it only takes on a finite number of values
 and that for fixed $m$, the preimage of $[m,\infty)$ is Zariski-closed.

Given fixed $t$ and $d$, there exists  $m_{t,d}\in \langle 1,t\rangle_\Q$
such that $f\in \C[x,y]$, $\vv[f]{\xi}{t}\ge m_{t,d}$ implies $f \in (x,y)^{d+1}$ independently on $\xi$
(by unloading, or using the definition \eqref{monomialdef}).
Thus
\[
\mm[d]{\xi}{t} < m_{t,d}
\]
for all $\xi$.

Similarly, there exists $i_{t,d}$ such that no $f\in \C[x,y]_d$
has a proper transform going through any center $p_i$ of $\vv{\xi}{t}$ with $i>i_{t,d}$.
Therefore for every $f\in \C[x,y]_d$, the value $\vv[f]{\xi}{t}$ belongs to the finite
set
\[
\left(\bigoplus_{i=1}^{i_{t,d}} \N v_i\right) \cap [1,m_{t,d})\ ,
\]
and the  $\mm[d]{\xi}{t}$ belong to this set.

Now let $V$ be  the $\C$-subspace of $\C[\theta,x,x^t]$ consisting of polynomials $P$ with $\deg_{\theta}(P)\le d$  and
$\deg_{x}(P)<m_{t,d}$. The space $V$ is obviously finite-dimensional,  $V\cong \C^N$ after taking the basis given by monomials.

Consider the composition of the substitution map
\[
{\mathcal J}_{k}\times \C[x,y]_d\rightarrow \C[\theta][[x,x^t]]\ ,
\]
given by $(\xi,f)\mapsto f(x,\xi(x)+\theta x^t)$,
with truncation $ \C[\theta][[x,x^t]]\rightarrow V$, seen
as an algebraic morphism of $\C$-schemes.

For each value $m$, the `incidence' subset
\[
\{(\xi,f)\in{\mathcal J}_{k}\times \C[x,y]_d \,|\, \vv[f]{\xi}{t}\ge m\}
\]
is by definition the preimage of the Zariski-closed set
\[
\{ \eta \in  V\,|\,\ord_x( \eta(x))\ge m\}
\]
hence Zariski-closed. It is also closed under scalar multiplication on the second component,
so it determines a closed subset
$I_m\subset{\mathcal J}_{k}\times \P(\C[x,y]_d)$.

The locus in ${\mathcal J}_{k}$ where
$ \mm[d]{\xi}{t}\ge m$ is the projection
of $I_m$ to ${\mathcal J}_{k}$,
therefore it is Zariski-closed.
\end{proof}

\begin{Pro}\label{continuoust}
 For every $\xi(x)$, the function $t\mapsto \mm{\xi}{t}$
 (for $t \in [1,\infty)$) is Lipschitz continuous
 with Lipschitz constant 1.
\end{Pro}

\begin{proof}
 For every $f\in \C[x,y]$, the function $t\mapsto \vv[f]{\xi}{t}$
 is a tropical polynomial function of degree at most $\deg(f)$.
 Therefore, the scaled function $\mu_f:t\mapsto \vv[f]{\xi}{t}/\deg(f)$
 is continuous concave and piecewise affine linear with slopes in
 $\{0,1/\deg(f),2/\deg(f), \dots, 1\}$
 (compare with \cite[Corollary C]{BFJ}).
 In particular, it is Lipschitz continuous with Lipschitz constant at most 1.

 The function $t\mapsto \mm{\xi}{t}$ in the claim is
 $\sup_{f\in \C[x,y]}\{\mu_f\}$; therefore it is also Lipschitz continuous
  with Lipschitz constant at most 1 (and it is not hard to see that it is actually
  equal to 1).
 \end{proof}

\begin{Rem}
 We proved in Proposition~\ref{semicont} that for a fixed $t$,
 very general series $\xi(x)$ give the same, minimal, value
 $\mm{\xi}{t}$ which we denote $\hat\mu(t)$.
\end{Rem}
\vskip\baselineskip

 By the countability of the rational number field, it follows that very general series
 $\xi(x)$ give the same (minimal) function $\mm{\xi}{t}$ of $t\in \Q$.
 Continuity of the functions $\mm{\xi}{t}$ then imply
 that very general series give the same function over all of $\R$, and
 also the following:

\begin{Cor}
  The function $t\mapsto \hat\mu(t)$ is Lipschitz continuous with
Lipschitz constant 1.
 \end{Cor}

It is immediate to extend the definition of $\mu$ and $\hat\mu$ to the
tree $\T$ of all valuations centered at $p_1$. The continuity properties
of the resulting function $\hat\mu:\T\rightarrow \R$
---which we shall not need---
are summarized as follows:

 \begin{Teo}\label{lower-semicontinuous}
    The function  $\hat\mu:\T\rightarrow \R$ is
lower semicontinuous for the weak topology and continuous
for the strong topology.
 \end{Teo}

 \begin{proof}
As in the proof of propositon \ref{continuoust},
for all $f\in \C[x,y]$, let $\mu_f(v)=v(f)/\deg(f)$.
By definition of the weak topology, $\mu_f$ is continuous
for all $f$.
Then, $\hat\mu(v)= \sup_{f\in \C[x,y]}\{\mu_f(v)\}$,
as the supremum of a family of continuous functions,
 is lower semicontinuous.

In order to prove continuity for the strong topology,
one needs to show continuity along all arcs in the profinite tree $\T$.
  It is not hard to see that (with minor changes)
the proof of propositon \ref{continuoust} works for series
$\xi$ with rational exponents as in Remark \ref{tree},
showing the desired continuity.

Alternatively, given a strong neighbourhood $U$ of a given valuation $v_0$,
there is a model of the plane in which every $v\in U$ is quasimonomial.
Then Proposition \ref{continuoust} shows that $\hat\mu$ is
continuous in $U$.
 \end{proof}
\vskip\baselineskip

The next claim will show the first analogy to Nagata's conjecture.

\begin{Pro}
\label{squares}
 If $t$ is the square of an integer, then
 a very general quasimonomial valuation
 $\vv{\xi}{t}$ is minimal.
\end{Pro}

\begin{proof}
For integral values of $t$, the cluster of centers of $\vv{\xi}{t}$
consists of the first $t$ points infinitely near to the origin
along the branch $y=\xi(x)$, and for each integer $m=qt+r$ (with $0\le r<t$)
\[
\mathcal I_{m}=(\pi_K)_*(\O_{S_K}(-q(E_{1}+\dots+E_t)-(E_1+\dots+E_r)))\ .
\]

For  $d>0$ and very general $\xi$, we want to prove that
$\mm[d]{\xi}{t}\le d\sqrt{t}$ or, in other words, that for every integer
$m>d\sqrt{t}$, the valuation ideal
$\mathcal I_{m}$ has no sections of degree $d$:
\[
H^0(\O_{S_K}(dL-q(E_{1}+\dots+E_t)-(E_1+\dots+E_r))) = 0\ ,
\]
where $L$ denotes the pullback of a line to $S_K$.
By semicontinuity (Proposition~\ref{semicont})
it will be enough to see this 
for a particular choice of $\xi$, e.g., an
irreducible polynomial of degree $a=\sqrt t$.
But the proper transform on $S_K$ of the projectivized curve
\[
D \colon YZ^{a-1}=Z^a\xi(X/Z)
\]
defined by $\xi$ is then an irreducible curve of
self-intersection zero, therefore nef, and
\[
D\cdot (dL-q(E_{1}+\dots+E_t)-(E_1+\dots+E_r))) = d\sqrt t-m < 0\ . \qedhere
\]
\end{proof}

\section{Anticanonical surfaces}
\label{sec:anticanonical}

This section contains a complete description of
the Mori cone of $S_K$ for $v=\vv{\xi}{t}$ with $t\le 7$
(see Theorem \ref{polyhedral} and Proposition \ref{polyhedral-integers}),
and substantial information for $t=7+\frac{1}{n_2}$, $n_2\in \N$
(see Proposition \ref{no-2} and Corollary \ref{MoriCone}).
In these cases
the rational surface $S_K$ obtained by blowing up the
cluster of  centers of a valuation $v$ on the plane is \emph{anticanonical},
meaning it has an effective anticanonical divisor. Under
this hypothesis, adjunction becomes a very powerful tool to study the
geometry of $S_K$.

We begin by justifying that $S_K$ is anticanonical in these cases.

\begin{Pro}\label{anticanonicalt}
Let $\vv{\xi}{t}$ be a divisorial quasimonomial valuation (so $t$ is rational),
and $S_K$ the blowup of its cluster of centers.
Let $A=[1,7] \cup \{7+\frac{1}{n}\}_{n\in \N} \cup \{ 9 \} \subset \R$.
\begin{enumerate}
 \item If $t \in A$, then $S_K$ is anticanonical.
 \item If $S_K$ is anticanonical for very general $\xi$, then $t\in A$.
\end{enumerate}
\end{Pro}

\begin{proof}
The question is whether the anticanonical class
$-\kappa=3L-\sum E_i$ on $S_K$ (where $L$ denotes the pullback of a line)
has nonzero global sections.

Suppose $t$ is an integer. Then
$K$ consists of $t$ free points; if $t\le9$,
there is a cubic going through them all, so $-\kappa$ is effective.
On the other hand, for an integer $t>9$,
there is no such plane cubic for general $K$.
Thus (1) and (2) hold when $t$ is an integer.

Now suppose $t=n_1+\frac{1}{n_2}$ is a nonintegral rational.
Then $K=(p_1,\dots,p_{n_1+n_2})$
has $n_1+1$ free centers and $n_2-1>0$ satellites,
all of them proximate to $p_{n_1}$; so
$\widetilde E_{n_1}=E_{n_1}-E_{n_1+1}-\dots-E_{n_1+n_2}$.
A simple  unloading computation (see Remark \ref{UnloadRem}
and \ref{ClusterDivRem}) then shows that
\begin{eqnarray*}
 H^0(\O_{S_K}(-\kappa)) & = & H^0(\O_{S_K}(3L-\sum E_i)) \\
& \cong & H^0(\O_{S_K}(3L-2E_1-(E_2+\dots+E_{n_1}))
\end{eqnarray*}
(the divisors $\tilde E_{n_1}$, $\tilde E_{n_1-1}$, \dots, $\tilde E_1$
intersect negatively and have been subtracted).
Consequently,  $S_K$ is anticanonical exactly when there exists a cubic
singular at $p_1$ and going through the free points $p_2$, \dots, $p_{n_1}$.
(If $n_1>1$ then both $p_2$ and $p_3$ are free, so if the cubic is irreducible
its singularity is a node.)
For $n_1\leq 7$, there is always such a cubic, so (1) holds, while
for a general choice of the free points we must have $n_1\le 7$, so (2) holds.

If the continued fraction for $t$ has more than 2 coefficients $n_1, n_2, \dots, n_r$,
the corresponding unloading computation
consists in subtracting, for $i=r-1, r-2, \dots, 1$, the divisors
$\tilde E_{n_1+\dots+n_i}$, $\tilde E_{n_1+\dots+n_i-1}$, \dots
$\tilde E_{n_1+\dots+n_{i-1}+1}$, and leads to
$H^0(\O_{S_K}(3L-2E_1-(E_2+\dots+E_{n_1+1})),$ so
$S_K$ is anticanonical exactly when there exists a cubic
singular at $p_1$ and going through the free points $p_2$, \dots, $p_{n_1+1}$.
Such a cubic always exists if $n_1\leq 6$, so (1) holds,
and for a general choice of the free points, we must have $n_1+1\le 7$, so (2) holds.
\end{proof}

The next lemma is needed for the proof of Proposition \ref{polyhedral-integers}.

\begin{Lem}\label{Blemma}
  Let $\vv{\xi}{t}$ be a divisorial quasimonomial valuation (so $t$ is rational),
and $S_K$ the blowup of its cluster of centers.
Let $B=[1,3] \cup \{3+\frac{1}{n}\}_{n\in \N} \cup \{ 4,5 \} \subset \R$.
\begin{enumerate}
 \item If $t \in B$, then $-\kappa-L$ is effective on $S_K$.
 \item If $-\kappa-L$ is effective on $S_K$ for very general $\xi$,
then $t\in B$.
\end{enumerate}
\end{Lem}

\begin{proof}
The proof is similar to the one for Proposition
\ref{anticanonicalt}. The integer cases are well known and easy to see.

Next say $t=n_1+\frac{1}{n_2}$ is a nonintegral rational.
Then $(-\kappa-L)\cdot \widetilde E_{n_1}<0$, so unloading
(as in the proof of Proposition \ref{anticanonicalt}) gives
$H^0(\O_{S_K}(-\kappa-L))
\cong H^0(\O_{S_K}(-\kappa-L-\widetilde E_{n_1})
\cong H^0(\O_{S_K}(2L-2E_1-(E_2+\dots+E_{n_1}))\ .$
The class of the proper transform of the line
through $p_1$ in the direction of $p_2$ is $\widetilde L=L-E_1-\cdots-E_i$
for some $2\leq i\leq n_1+1$. Thus
$(2L-2E_1-(E_2+\dots+E_{n_1}))\cdot \widetilde L<0$.
Therefore, if $n_1\leq 3$, subtracting $\widetilde L$ and unloading we have
$H^0(\O_{S_K}(2L-2E_1-(E_2+\dots+E_{n_1}))) \cong H^0(\O_{S_K}(L-E_1-E_2))\neq 0$.
Thus (1) holds for such $t$,
while for a general choice of the free points we must have $n_1\le 3$,
so (2) holds for such $t$.

Finally, if the continued fraction for $t$ has more than 2 coefficients $n_i$,
the corresponding unloading computation leads to
$H^0(\O_{S_K}(-\kappa-L))
\cong H^0(\O_{S_K}(2L-2E_1-(E_2+\dots+E_{n_1+1}))\ .$
Again $\widetilde L=L-E_1-\cdots-E_i$
for some $2\leq i\leq n_1+1$, so subtracting $\widetilde L$ and unloading
gives
$H^0(\O_{S_K}(2L-2E_1-(E_2+\dots+E_{n_1+1}))) \cong
H^0(\O_{S_K}(L-E_1-\cdots-E_{n_1+2-i}))$.
The latter is clearly nonzero if $n_1\leq 3$, so (1) holds, and
for a general choice of the free points, we must have $n_1\le 2$,
so (2) holds.
\end{proof}

\begin{Rem}
Note that if $t\le 7$, then $K$ has at most 7 free centers, so there is always
a divisor $\widetilde \Gamma$ in $|3L-2E_1-\sum_{i>1, p_i\text{ free}} E_i|$.
For general $\xi$,
$p_1, p_2, p_3$ are not aligned and $p_1, \dots, p_6$ do not belong to a conic,
so $\widetilde\Gamma$ can be assumed to be the proper transform of
an irreducible nodal cubic $\Gamma$, and
$\Gamma_K=\widetilde \Gamma + \sum \widetilde E_i$ on $S_K$
is a particular anticanonical divisor
which contains all exceptional components (independently of $t$).
For nongeneral $\xi$, $\widetilde \Gamma$ may be reducible, but
$\Gamma_K=\widetilde \Gamma + \sum \widetilde E_i$
still determines an effective anticanonical divisor which
contains all exceptional components.
\end{Rem}

\begin{Teo}\label{polyhedral}
 Let $\vv{\xi}{t}$ be a divisorial quasimonomial valuation with $t\le 7$,
and $S_K$ the blowup of its cluster of centers.
Let $s$ be the number of centers.
Then the number of $(-1)$-curves other than $E_s=\widetilde E_s$
is at most $s$, and $\Mor(S_K)$ is a polyhedral cone,
spanned by the classes of the $\widetilde E_i$,
$\widetilde \Gamma$ and the $(-1)$-curves,
where $\Gamma$ is a nodal cubic as above.
\end{Teo}

\begin{proof}
 Let 
 $\Gamma_K$ be an effective  anticanonical divisor containing all exceptional
 components $\widetilde E_i$; for general $\xi$ we can write
 $\Gamma_K=\widetilde \Gamma + \sum \widetilde E_i$, where $\Gamma$ is a nodal cubic.
 Particular cases in which the cubic is reducible are treated similarly and
 we leave the details to the reader. We claim that
 every irreducible curve $C\subset S_K$ which is not a component of
 $\Gamma_K$ lies in $\Mor(S_K)^\preceq$. Indeed, $C$ is the proper transform
 of a curve $\pi_K(C)\subset \P^2$; if $\pi_K(C)$ does not go through the
 origin $p_1$ of $K$, then $C$ intersects $\widetilde \Gamma$ and so
\[
 C\cdot \kappa=-(C\cdot(\Gamma_K))= -(C\cdot \widetilde \Gamma)<0\ .
\]
Otherwise, $C$ intersects some $\widetilde E_i$ and so
\[
C\cdot \kappa = -(C\cdot(\Gamma_K))\le -(C\cdot \widetilde E_i)<0\ .
\]
Thus by Mori's cone theorem, $\Mor(S_K)$ is generated by the rays spanned
by the components of $\Gamma_K$ and the $(-1)$-curves,
so it only remains to bound the number of $(-1)$-curves.

But a $(-1)$-curve $C$ satisfies $C\cdot \kappa=-1$,
so if it is not a component of $\Gamma_K$, it must intersect it in exactly one component.
Write $C=dL-\sum m_i E_i$. If $C$ meets only $\widetilde E_k$,
it must satisfy $m_j=\sum_{p_i\succ p_j} m_i$ (i.e., $C\cdot \widetilde E_j=0$) for all $j\ne k$,
$m_k=\sum_{p_i\succ p_k} m_i+1$ (i.e., $C\cdot \widetilde E_k=1$) and
$3d-\sum m_i=1$ (i.e., $C\cdot \Gamma_K=1$). These are $s+1$ linearly independent conditions which uniquely determine the class of $C$; so there is at most
one $(-1)$-curve meeting $\widetilde E_k$. On the other hand,
$C$ cannot meet only $\widetilde\Gamma$,
because then $C\cdot \widetilde E_j=0$ for all $j$, which implies
$m_j=C\cdot E_j=0$ for all $j$, and hence $1=C\cdot \Gamma_K=3d-\sum m_i=3d$.
Thus the number of $(-1)$-curves not
components of $\Gamma_K$ is at most $s$.
\end{proof}

\begin{Rem}
  Along the way we proved that there are finitely many
curves with negative selfintersection when $t\le 7$.
Indeed, if $C$ is such a curve, and it is not a component
of $\Gamma_K=\widetilde \Gamma + \sum \widetilde E_i$
then $C\cdot \kappa< 0$, which implies
$0>C^2+C\cdot \kappa=2g-2$, so $C$ is a rational curve
and in fact a $(-1)$-curve, of which there are
at most $s$.
\end{Rem}
\vskip\baselineskip

For $t\in B$, one can be a bit more precise: not only do the
negative curves generate the Mori cone over $\R$,
they generate the monoid of effective classes (over $\N$).

\begin{Pro}\label{polyhedral-integers}
 Let $\vv{\xi}{t}$ be a divisorial quasimonomial valuation with $t\in B$,
$t>1$, and $S_K$ the blowup of its cluster of centers.
Let $s$ be the number of centers.
Then the monoid in $\Pic S_K\cong \Z^{s+1}$ of the effective classes
has a minimal (finite) set of generators consisting of
the classes of the $\widetilde E_i$, the $(-1)$-curves, and
the components of $-\kappa-L$ meeting
$-\kappa-L$ negatively.
\end{Pro}

\begin{proof}
Thanks to Lemma \ref{Blemma}, we can apply \cite[Proposition III.ii.1]{Har98}.
\end{proof}

\begin{Rem}\label{n1=4}
A similar result holds for any divisorial quasimonomial valuation $\vv{\xi}{t}$ when $t=4+\frac{1}{n_2}$,
namely the effective monoid in $\Pic S_K$ is generated by $\widetilde E_i$, $i=1,\ldots, s=n_1+n_2$, and the proper transform
$\widetilde E_0$ of $L-E_1-E_2$ (and $Q=2L-E_1-\cdots-E_5$ if $\widetilde E_0=L-E_1-E_2$, i.e., if $p_3$ does not belong to the line through
$p_1$ in the direction of $p_2$). We sketch the argument
in case $\widetilde E_0=L-E_1-E_2$.
Take the basis $D_0,\ldots,D_s$ for the divisor class group of $S_K$,
satisfying $D_i\cdot \widetilde E_j=\delta_{ij}$, where
$\delta_{ij}$ is Kronecker's delta, hence $\delta_{ij}=0$ for $i\neq j$ and 1 if $i=j$.
(Thus $D_i$ is just the basis dual to $\widetilde E_j$, specifically:
$D_0=L$,
$D_1=L-E_1$,
$D_2=2L-E_1-E_2$,
$D_3=2L-E_1-E_2-E_3$,
$D_4=2L-E_1-E_2-E_3-E_4$,
$D_5=Q=2L-E_1-E_2-E_3-E_4-E_5$,
$D_6=4L-2E_1-2E_2-2E_3-2E_4-E_5-E_6$,
$D_7=6L-3E_1-3E_2-3E_3-3E_4-E_5-E_6-E_7$,
$D_8=8L-4E_1-4E_2-4E_3-4E_4-E_5-E_6-E_7-E_8$,
and so on, so for $4<i\leq s-4$ we have
$D_{i+4}=2iL-iE_1-iE_2-iE_3-iE_4-E_5-\ldots-E_{i+4}$.)
Every prime divisor $D$ not among the $\widetilde E_j$ is $\sum_i(D\cdot \widetilde E_i)D_i$, hence
it suffices to check the divisors $D_i$.
It is easy to write down the classes $D_i$ explicitly and then to check that each
$D_i$ is a nonnegative integral sum of classes $\widetilde E_j$, $j\geq 0$, when $i<5$,
and a nonnegative integral sum of the classes $Q$ and $\widetilde E_j$, $j\geq 1$, when $i\geq 5$.
(Essentially the same argument works when $\widetilde E_0=L-E_1-\cdots-E_l$
for $l>2$, except the result is that
$D_i$ is a nonnegative integral sum of the classes $\widetilde E_j$, $j\geq 0$,
for all $i$. In this case we note that $Q\cdot \widetilde E_0<0$ so $Q$ is no longer prime,
and $\widetilde E_0\cdot \widetilde E_4$ having to be nonnegative forces
$l\leq 5$.)

In fact, we can show that
a similar result holds for $t=n_1+\frac{1}{n_2}$ also for $n_1=5$ and $6$, namely
that there are only finitely many prime divisors of negative self-intersection on $S_K$,
and they generate the effective monoid. The proof is more involved, however,
since, for $n_1=5$, $D_6=2L-E_1-\cdots-E_6$ need not be effective
but it could be and if it is, it may but might not be a prime divisor.
Likewise, for $n_1=6$, additional cases arise:
$3L-2E_1-E_2-\cdots-E_7$ and $5L-2E_1-\cdots-2E_6-E_7-E_8$ are effective
but may or might not be prime, and
$2L-E_1-\cdots-E_6$ and $2L-E_1-\cdots-E_7$
may or might not be effective.
Nonetheless, the proof follows similar lines (in each of the several
cases, find an explicit finite set of
generators for the effective monoid, and then show each generator
is a sum of negative curves). Because checking the various cases is somewhat lengthy,
we do not include the proof here.
\end{Rem}
\vskip\baselineskip

For $7<t<8$, it is not clear which values of $t$ give polyhedral Mori cones,
but C.~Galindo and F.~Monserrat \cite{GM05} give some positive
results in this context. In particular, their Corollary 5, (1)
shows that for $t=7+1/n_2$ with
$n_2=1,2,\dots,8$,  $\Mor(S_K)$ is polyhedral. We show this result
is sharp, in the sense that $\Mor(S_K)$ is not polyhedral for $n_2>8$,
provided that $\xi$ is very general (see Corollary \ref{MoriCone}).
On the other hand, parts (2) and (3) of \cite[Corollary 5]{GM05}
are sharpened by Theorem \ref{polyhedral} above.

In preparation for proving Corollary \ref{MoriCone}, we first prove
 a result concerning prime divisors $C$ with $C^2<-1$.

\begin{Pro}\label{no-2}
 Let $\vv{\xi}{t}$ be a very general divisorial quasimonomial valuation with $t=7+1/n_2$
 for $n_2\ge1$, and let $S_K$ be the blowup of its cluster of centers.
The only prime divisors $C$ in $S_K$ with $C^2\le -2$ are components of the exceptional
divisors $E_i$.
\end{Pro}

\begin{proof}
As before, let $\Gamma$ be  a nodal cubic curve which has its node
at the origin and goes through six additional free centers, $p_2, \dots, p_7\in K$.
Then  $\Gamma_K=\widetilde \Gamma + \sum_{i=1}^7 \widetilde E_i$ on $S_K$
is the unique effective anticanonical divisor.

By adjunction we have $C^2+\kappa_{S_K}\cdot C=2g-2$, so
$C^2<-2$ implies $\Gamma_K\cdot C<0$, hence
$C$ is a  component of $\Gamma_K$.
Computing the self-intersection of each of them
shows that the only possibility is $C=\widetilde E_7=E_7-E_8-\cdots-E_{s}$
where again $s$ is the total number of blowups and $C^2=-1-n_2$.

By adjunction again, if $C^2=-2$, then $C$ is rational
and $\kappa_{S_K}\cdot C=0$, i.e., it is a $(-2)$-curve.
Thus the question is what $(-2)$-curves can occur on $S_K$.
The exceptional components $\widetilde E_i$ for $i\ne 7, s$
are $(-2)$-curves. Now assume that $C$ is not one of them.
Then $\Gamma_K\cdot C=0$ implies
$C\cdot \widetilde E_i=0$ for $i=0,\ldots,7$, and $C\cdot \widetilde E_i\geq0$ for $i>7$.

Write $C=dL-m_1E_1-\cdots-m_sE_s$. The constraint $C\cdot \widetilde E_7=0$ gives
$m_7=m_8+\cdots+m_s$. The constraints $C\cdot \widetilde E_i=0$ for $i=1,\ldots,6$ give
$m_1=\cdots=m_7$. Taking $m=m_1$, $C\cdot \widetilde \Gamma=0$ gives
$3d=7m+m_8+\cdots+m_s=8m$, so $d=8m/3$. Note that $d$ is an integer.

Consider the case that $n_2=1$. Then $-2=C^2=(8m/3)^2-8m^2=-8m^2/9$.
This has no integer solutions, so no $C$ exists.

Next consider the case that $n_2=2$, so $s=9$. The possible solutions
$C$ to $C^2=-2$, $C\cdot \kappa_{S_K}=0$ with $C\cdot L\geq0$ are known
(see the second half of the proof of \cite[Proposition 25.5.3]{Man86}); they are:
$(E_i-E_j)-s\kappa_{S_K}$ with $1\leq i,j\leq 9$, $i\neq j$, $s\geq0$;
$(L-E_i-E_j-E_k)-r\kappa_{S_K}$ with $1\leq i,j,k\leq 9$, $i, j, k$ distinct, $r\geq0$;
$(2L-E_{i_1}-\cdots-E_{i_6})-r\kappa_{S_K}$ with $1\leq i_j\leq 9$, $i_j$
distinct for $1\leq j\leq 6$, $r\geq0$; and
$(3L-2E_{i_1}-E_{i_2}-\cdots-E_{i_8})-r\kappa_{S_K}$,
$1\leq i_j\leq 9$, $i_j$ distinct for $1\leq j\leq 8$, $r\geq0$. An exhaustive check shows that each of these divisors
intersects some exceptional component or $\Gamma$ negatively, and thus
is either itself a component of an exceptional curve, or is not reduced or irreducible.

Now consider the case that $n_2\geq3$, so $s\geq10$, and we can write
$C=dL-m(E_1+\cdots+E_7)-m_8E_8-\cdots-m_sE_s=
(8m/3)L-m(E_1+\cdots+E_7)-m_8E_8-\cdots-m_sE_s$.
Let $m=3b$, so $C=8bL-3b(E_1+\cdots+E_7)-m_8E_8-\cdots-m_sE_s$.
Then $\Gamma_K\cdot C=0$ gives  $3b-m_8-\cdots-m_s=0$ and $C^2=-2$
gives $b^2-m_8^2-\cdots-m_s^2=-2$. Numerical considerations no longer suffice;
there are many solutions to $3b-m_8-\cdots-m_s=0$ and $b^2-m_8^2-\cdots-m_s^2=-2$.
For example, we have $C=8L-3(E_1+\cdots+E_7)-E_8-E_9-E_{10}$
(i.e., $s=10$, $n_2=3$, $b=1$, and $m_8=m_9=m_{10}=1$).
The following lemma however shows that such $C$ can not be the
class of a prime divisor, and finishes the proof.
\end{proof}

\begin{Lem}
\label{lemmagamma}
Let $S_K$ be as in Proposition \ref{no-2}. Then there is no prime divisor $C$ on $S_K$
with $C\cdot \kappa_{S_K}=0$ other than $\widetilde E_i$ for $i\neq 7$.
\end{Lem}

\begin{proof}
By the end of the proof of Proposition \ref{no-2}, if such a $C$ exists it must be
$C=dL-m(E_1+\cdots+E_7)-m_8E_8-\cdots-m_sE_s$, where
$d=8m/3$, $m=3b=m_8+\cdots+m_s$ for some $b$, and $m_8\geq \cdots\geq m_s>0$.
The divisor class $8L-3(E_1+\cdots+E_7)$ is effective and base point free,
and has irreducible global sections;
in fact it is the class of a homaloidal net, see Proposition \ref{cremona8} below.
In particular it is nef. Pick an irreducible $B\in |8L-3(E_1+\cdots+E_7)|$.
Since $B\cdot \widetilde \Gamma=B\cdot \widetilde E_i=0$ for $i<7$,
we see $B|_{\Gamma_K}$ is a divisor which vanishes on each component
$\widetilde E_i$, $i<7$ of $\Gamma_K$, and consists of a divisor $B'$
of degree $3$ on the interior
of component $\widetilde E_7$. Since $E_i|_{\Gamma_K}=E_s|_{\Gamma_K}$ for $i\geq 8$ and $E_i$ is disjoint from
$\widetilde E_j$ for $i\geq 8$ and $j<7$, we see $(-m_8E_8-\cdots-m_sE_s)|_{\Gamma_K}$
is a divisor which is trivial on each component of $\Gamma_K$ except $\widetilde E_7$, and on
$\widetilde E_7$ it gives the divisor $(m_8+\cdots+m_s)p_8=mp_8=3bp_8$.
Thus $\O_{\Gamma_K}(C)$ is the same as $\O_{\Gamma_K}(bB'-3bp_8)$.

Consider the restriction exact sequence
$$0\to\O_{S_K}(C-\Gamma_K)\to\O_{S_K}(C)\to\O_{\Gamma_K}(C)\to0.$$
Then, since $C$ is by assumption a prime divisor, we have
$h^0(S_K, \O_{S_K}(C-\Gamma_K))<h^0(\O_{S_K}(C))$,
which by taking cohomology of the short exact sequence
implies $h^0(\O_{\Gamma_K}(bB'-3bp_8))>0$.
But $\deg(bB'-3bp_8)=0$ so $h^0(\O_{\Gamma_K}(bB'-3bp_8))>0$
implies $bB'-3bp_{8}\sim0$ (where $\sim$ denotes linear equivalence).
Since the class $B'$ is fixed of positive degree but $p_8$ is very general, this
would imply that $3b(p-q)$ for every pair of interior points $p,q\in \widetilde E_7$,
contradicting the fact that the identity component of $\Pic(\Gamma_K)$
is isomorphic to the multiplicative group $\C$ of the ground field
(and so not every element is a torsion element).
Thus there is no such prime divisor $C$.
\end{proof}

\begin{Rem} When $8\leq s\leq 15$, it is enough for $p_8$ to be a general,
rather than very general,
point of $\widetilde E_7$ in order to conclude that $S_K$ has no $(-2)$-curves other than those
arising as components of the exceptional loci of the points blown up.
To see this, consider a prime divisor $C\subset S_K$ such that $K_{S_K}\cdot C=0$
and $C\cdot L>0$. Write $C\sim dL-m_1E_1-\cdots-m_sE_s$. Then, as above,
$C=dL-m(E_1+\cdots+E_7)-m_8E_8-\cdots-m_sE_s
=b(8L-3(E_1+\cdots+E_7))-m_8E_8-\cdots-m_sE_s$
and $m=m_8+\cdots+m_s$,
so
$$-2=C^2=b^2-m_8^2-\cdots-m_s^2
\leq b^2-\frac{m^2}{(s-7)^2}(s-7)=b^2\frac{s-16}{s-7},$$
hence for $8\leq s\leq 15$ we have
$$d^2=8b^2\leq 8\frac{2s-14}{16-s}.$$
Thus for $8\leq s\leq 15$ we have
$d^2\leq 128$, so $d\leq 11$.

I.e., for $8\leq s\leq 15$ we see that $d$ is bounded (i.e., $C\cdot L\leq 11$)
and hence that there are only finitely many possible $(-2)$-classes $C$.
Since it is only for these classes that we must avoid $C|_{-K_X}=0$ in order for
$C$ not to be effective, it is enough for $p_8$ to be general, in order to know that
every $(-2)$-class is a component of the exceptional locus of a blow up.
\end{Rem}

\begin{Cor}
\label{MoriCone}
 Let $\vv{\xi}{t}$ be a very general divisorial quasimonomial valuation with $t=7+1/n_2$
 for $n_2\ge1$, and let $S_K$ be the blowup of its cluster of centers.
Then $\Mor(S_K)$ is a cone with at most countably many extremal rays,
spanned by the classes of the $\widetilde E_i$,
$\widetilde \Gamma$ and the $(-1)$-curves,
where $\Gamma$ is a nodal cubic as above.
Moreover, when $n_2>8$, there are infinitely many
$(-1)$-curves.
\end{Cor}

\begin{proof}
 Because of Mori's theorem, and because every divisor $C$ in $\Mor(S_K)^\succcurlyeq$
 either is a component of $\Gamma_K=\widetilde \Gamma + \sum_{i=1}^7 \widetilde E_i$
 or satisfies $C\cdot \Gamma_K=0$, it is enough by Proposition \ref{no-2}
to show that the only prime
 divisors with $C\cdot \Gamma_K=0$ are the $(-2)$-curves of the form $\widetilde E_i$.
 But this follows from Lemma \ref{lemmagamma}.

There will indeed be infinitely many extremal rays when
$n_2\geq 9$, because in this situation there are
infinitely many $(-1)$-curves $C$.
Briefly, we reduce to the case that $S_K$ is the blow up
of a cluster of 9 infinitely near points coming from blowing up
9 times at a very general point of a nodal cubic.
In this situation, the only restrictions for a divisor $C$ with $C^2=C\cdot \kappa=-1$
to be a $(-1)$-curve follow from the proximity inequalities, which impose restrictions
only to the monotonicity of the multiplicities of $C$ at the centers of the blowups.

In more detail,
apply the degree 8 Cremona map $\Phi_8$ given by $|8L-3(E_1+\cdots+E_7)|$
(see Proposition \ref{cremona8}), which maps $S_K$ to $\P^2$,
mapping $E_7$ to a nodal cubic $\Gamma'$ and
representing $S_K$ as a blowup of $\P^2$ of two clusters of points.
One is a cluster of 7 points $p'_1,\ldots,p'_7$ on $\Gamma'$ infinitely near the
node, and the other is a cluster of $n_2$ points $p'_8,\ldots,p'_{7+n_2}$
on $\Gamma'$ infinitely near $p'_8$, which is a very general point
of $\Gamma'$. If $n_2\geq9$, the blowup of $p'_8,\ldots,p'_{16}$
gives a surface $S$ with infinitely many $(-1)$-curves. Blowing up the remaining
points $p'_i$ does not affect this, since none of the remaining points
$p'_i$ can be on any of the $(-1)$-curves on $S$. (This is because
the generality of $\vv{\xi}{t}$ causes every $(-1)$-curve
on $S$ except $E'_{16}$ to meet the proper transform $\widetilde \Gamma'$
at points not infinitely near to either $p'_1$ and $p'_8$.
For the fact that $S$ has infinitely many $(-1)$-curves,
using the notation of Remark \ref{cremona-1},
note that there are infinitely many classes
$C=dL'-m_8E'_8-\cdots-m_{16}E'_{16}$ with $C^2=C\kappa_S=-1$
such that $m_8\geq \cdots\geq m_{16}$,
where $\kappa_S$ is the canonical class of $S$.
In fact it is not hard to see that all $C$ with $C^2=C\kappa_S=-1$
are precisely the classes $C=E'_{16}+N+\frac{N^2}{2}\kappa_S$
where $N$ is an arbitrary class satisfying $N\cdot\kappa_S=0$ and $N\cdot E'_{16}=0$,
hence $N$ is any integer linear combination of $L'-E'_8-E'_9-E'_{10},
\widetilde E'_8=E'_8-E'_9, \ldots, \widetilde E'_{14}=E'_{14}-E'_{15}$.
Clearly there are not only infinitely many such $C$ but also infinitely many also satisfying
$m_8\geq \cdots\geq m_{16}$. Any divisor $D$ on $S$
with $D\cdot \kappa_S=0$ is by linear algebra an integer linear combination
of $L'-E'_8-E'_9-E'_{10}, \widetilde E'_8, \ldots \widetilde E'_{15}$.
If $D$ is in addition a prime divisor but not one of the $\widetilde E'_i$ nor $-\kappa_S$, then
$D$ is in the kernel of the functorial homomorphism
$\pi:\operatorname{Pic}(S)\to \operatorname{Pic}(\widetilde\Gamma')$, but the
expression of $D$ as a
linear combination of $L'-E'_8-E'_9-E'_{10},
\widetilde E'_8, \ldots, \widetilde E'_{15}$ must
involve $L'-E'_8-E'_9-E'_{10}$, which implies that
the image of $L'-E'_8-E'_9-E'_{10}$ under $\pi$ has finite order,
contradicting the cluster $p'_8,\ldots,p'_{7+n_2}$ being very general.
Thus the only prime divisors satisfying $D\cdot \kappa_S=0$
are $\widetilde E'_8, \ldots \widetilde E'_{15}$ and $-\kappa_S$.
It now follows by \cite[Proposition 3.3]{LH} that
every class $C=dL'-m_8E'_8-\cdots-m_{16}E'_{16}$ with $C^2=C\kappa_S=-1$
such that $m_8\geq \cdots\geq m_{16}$ is the class of a $(-1)$-curve.)
\end{proof}

\begin{Rem}
\label{cremona-1}
Here we explain the action of $\Phi_8$, used in the proof
of Corollary \ref{MoriCone}, in terms of the components
of $E_1$. With $s=7+n_2$, the components are
$\widetilde E_1=E_1-E_2$,
$\widetilde E_2=E_2-E_3$,
$\widetilde E_3=E_3-E_4$,
$\widetilde E_4=E_4-E_5$,
$\widetilde E_5=E_5-E_6$,
$\widetilde E_6=E_6-E_7$,
$\widetilde E_7=E_7-E_8-\cdots-E_s$,
$\widetilde E_8=E_8-E_9, \ldots,
\widetilde E_{s-1}=E_{s-1}-E_s$ and
$\widetilde E_s=E_s$.
Applying $\Phi_8$ is equivalent to blowing down
the $(-1)$-curve $3L-2E_1-E_2-\cdots-E_7$, followed by
$\widetilde E_1$,
$\widetilde E_2$,
$\widetilde E_3$,
$\widetilde E_4$,
$\widetilde E_5$, and
$\widetilde E_6$.
Under this blow down, $\widetilde E_7$ maps to a nodal cubic $C'$
whose node is the image of the contracted curves, while
$E_s$, $\widetilde E_{s-1}, \cdots, \widetilde E_8$ contract to a smooth point on this cubic.
Reversing this blow down gives a blow up of $\P^2$ at two clusters of points,
the first $p'_1,\ldots,p'_7$, and the second $p'_8,\ldots,p'_s$ where
$p'_1,p'_8\in\P^2$ and all of the points are free but lie on the proper transform of
$C'$. In terms of the exceptional divisors $E'_i$ of the centers $p'_i$ we have
$\widetilde E'_1=E'_1-E'_2=E_6-E_7$,
$\widetilde E'_2=E'_2-E'_3=E_5-E_6$,
$\widetilde E'_3=E'_3-E'_4=E_4-E_5$,
$\widetilde E'_4=E'_4-E'_5=E_3-E_4$,
$\widetilde E'_5=E'_5-E'_6=E_2-E_3$,
$\widetilde E'_6=E'_6-E'_7=E_1-E_2$,
$E'_7=\widetilde E'_7=3L-2E_1-E_2-\cdots-E_7$,
and also
$\widetilde E'_8=E'_8-E'_9=E_8-E_9, \ldots,
\widetilde E'_{s-1}=E'_{s-1}-E'_s=E_{s-1}-E_s$ and
$E'_s=E_s$.
We also have
$L'=8L-3E_1-\cdots-3E_7$, and
$\widetilde E_7=E_7-E_8-\cdots-E_s=3L'-2E'_1-E'_2-\cdots-E'_s$.
\end{Rem}

\section{A variation on Nagata's conjecture}
\label{sec:Nagata}

In this section we elaborate on the close analogy with Nagata's conjecture.

Let $K$ be a finite union of finite weighted clusters on $\P^2$,
and assume that the proximity inequalities
\[
 m_p\ge \sum_{q\succ p} m_q
\]
are satisfied, with the sum  taken  over all points
$q\in K$ proximate to $p$.

Then
\[
\H_{K,m} = \pi_* \left(\O_{S_K}\left(-\sum _{p\in K}m_pE_p\right)\right)
\]
is an ideal sheaf on $\P^2$ for which
\[
 h^0(\H_{K,m}(d)) = \frac{(d+1)(d+2)}{2}-\sum_{p\in K} \frac{m_p (m_p+1)}{2}
\]
for $d\gg 0$,  and its general
member defines a degree $d$ curve with multiplicity $m_p$
at each $p\in K$.

It is expected that, if $K$ is suitably general, then the dimension
count is correct as soon as it gives a nonnegative value:

\begin{Con}[Greuel-Lossen-Shustin, {\cite[Conjecture 6.3]{GLS}}]
\label{gls}
 Let $K$ be a finite union of weighted clusters on the plane, satisfying the
 proximity inequalities, and $\H_{K,m}$ the corresponding ideal sheaf.
 Assume that $K$ is general among all clusters with the same proximities,
 and let $d$ be an integer which is larger than the sum of the three biggest
 multiplicities of $m$. Then
  $$h^0(\H_{K,m}(d))=\max\left\{0,\frac{(d+1)(d+2)}{2}-\sum_{p\in K}
 \frac{m_p (m_p+1)}{2}\right\}.$$
\end{Con}

\begin{Pro}
 If the Greuel-Lossen-Shustin conjecture holds, then $\forall t\ge 9$
 a very general quasimonomial valuation
 $\vv{\xi}{t}$ is minimal.
\end{Pro}

\begin{proof}
 By continuity of $\hat\mu(t)$, it is enough to consider rational $t>9$.
 Let $K=(p_1,\dots,p_s)$ be the sequence of centers, with weights
 $(v_1,\dots,v_s)$. 
 For each integer $k>0$, set $m_k=kt/ v_s$.
 We shall prove that there is a sequence of integers $d_k$
with $m_k >d_k\sqrt{t}$ and $\lim_{k\to \infty}m_k/d_k=\sqrt t$
such that if $\xi$ is very general, then the valuation ideal
$\mathcal I_{m_k}$ has no sections of degree $d$.
It will follow that $\mm{\xi}{t}\le \lim_{k\to\infty} m_k/d\le \sqrt t$ and
$\vv{\xi}{t}$ is minimal.

By Lemma \ref{simple}, the ideal
$\mathcal I_{m_k}=(\pi_K)_*(\O_{S_K}(-\sum\bar m_iE_i))$
is simple and the three largest multiplicities
are $\bar m_1=\bar m_2=\bar m_3=k/v_s$.
Hence $\bar m_1+\bar m_2+\bar m_3=3k/v_s<\sqrt{t}k/v_s$.
Without loss of generality we may assume that $k$ is large enough
that there exist
integers $d_k<m_k/\sqrt t$ which also satisfy $d_k> \bar m_1+\bar m_2+\bar m_3$.
In this case the hypothesis in conjecture \ref{gls} is satisfied and
$h^0(\H_{K,m}(d_k))=\max\left\{0,(d_k+1)(d_k+2)/2-\sum
\bar m_i (\bar m_i+1)/2\right\}.$
By way of contradiction, assume $\mathcal I_{m_k}$ has sections of degree $d_k$.
Then  $(d_k+1)(d_k+2)/2\ge\sum \bar m_i (\bar m_i+1)/2$,
which together with $d_k<m_k/\sqrt t=\sum \bar m_i^2$
implies $3d_k+2>\sum \bar m_i\ge10\,\bar m_1>kt/v_s=m_k$,
a contradiction.
\end{proof}

With this in mind, we propose the following:

\begin{Con}[Nagata's Conjecture for quasimonomial valuations]\label{nqmvconj}
For all  $t\ge 9$, we have
 $\hat\mu(t)=\sqrt{t}$.
\end{Con}

\begin{Pro}
 Conjecture~\ref{nqmvconj} implies Nagata's conjecture.
\end{Pro}

\begin{proof}
 Let $t> 9$ be a nonsquare integer.
 By a ``collision de front'' \cite{Hir85}
 and semicontinuity, Nagata's conjecture
 for $t$ points would follow by showing that, for a very general $\xi(x)\in \C[[x]]$,
 and for every couple of integers $d, m$ with $0<d<m\sqrt{t}$,
 the ideal $(x^t,y-\xi(x))^m\cap \C[x,y]$ has no nonzero element in degree $d$.
 But this is an immediate consequence of $\hat\mu(t)=\sqrt{t}$.
\end{proof}

In view of the computations in next section, we expect that
in fact the range of $t$ for which $\hat\mu(t)=\sqrt{t}$ is larger,
see Conjecture~\ref{strong}.

\section{Supraminimal curves}
\label{sec:supraminimal}

If some valuation $v$ is not minimal, this is due to the existence
of a curve $C$ (which may be taken irreducible and reduced) with
larger valuation than what one would expect from the degree.
These curves will be called \emph{supraminimal}, and are
the subject of this section. For simplicity, we fix $p_1=(0,0)\in \A^2\subset \P^2$
as before.

\begin{Lem}\label{negcurves}
If there is an  irreducible
 polynomial $f\in \C[x,y]$  with
\[
 \vv[f]{\xi}{t}>\frac{1}{\sqrt{\vol(\vv{\xi}{t})}}\deg(f)\ ,
\]
then  $\vv[f]{\xi}{t}=\mm{\xi}{t} \deg(f)$.

Moreover, if $\mm{\xi}{t}>\frac{1}{\sqrt{\vol(\vv{\xi}{t})}}$,
 then there is such an irreducible
 polynomial $f$.
 \end{Lem}

 In the case above we  say that $f$ computes $\mm{\xi}{t}$.

\begin{proof}
By continuity of $\mm{\xi}{t}$ as a function of $t$, it
 is enough to consider the case $t\in \Q$. Let $v=\vv{\xi}{t}$.

 Let $f$ be as in the claim, and $d=\deg f$.
It will be enough to prove that, for every polynomial $g$ with
degree $e$ and $v(g)=w>\frac{e}{\sqrt{\vol(v)}}$, $f$ divides $g$.
Choose an integer $k$ such that
$kw\in \N$ is an integer multiple of $t$, and consider the ideal
$$I_{kw}=\{h\in \C[x,y]\,|\,v(h)\ge kw\}.$$
A general $h\in I_{kw}$ has $kw/t$ Puiseux series roots, each of them
of the form $\xi(x)+ax^t+\dots$; therefore the local intersection multiplicity
of $h=0$ with $f=0$ is
\begin{equation}\label{genericintersection}
I_0(h,f)\ge\frac{kw}{t}v(f)>\frac{kwd}{t\sqrt{\vol(v)}}=\frac{kwd}{\sqrt{t}}.
\end{equation}
Since obviously $g^k\in I$, the intersection multiplicity $I_0(g^k,f)$ is
bounded below by \eqref{genericintersection}, and therefore
$$
I_0(g,f)>\frac{wd}{\sqrt{t}}=dw\sqrt{\vol(v)}>de,
$$
so $f$ is a component of $g$.

Now assume $\hat\mu(v)>\frac{1}{\sqrt{\vol(v)}}$. So there is a polynomial
 $g\in \C[x,y]$ of degree $e$ with $v(g)>\frac{e}{\sqrt{\vol(v)}}$.
 Since $v(f_1 \cdot f_2)=v(f_1)+v(f_2)$, it follows that
 at least one irreducible component $f$ of $g$,
 satisfies $v(f)>\frac{\deg f}{\sqrt{\vol(v)}}$.
\end{proof}

\begin{Pro}\label{genericneg}
 Assume that $d\in \N, m_1/n_1, \dots, m_r/n_r\in \Q$, with
 $\gcd\{m_i,n_i\}=1$ are such that, for
 a very general $\xi(x)$, there exists an irreducible $f\in \C[x,y]$ with
 $\deg(f)=d$ which decomposes in $\C[[x,y]]$ as a product
 of $r$ irreducible series $f=f_1\dots f_r$ with $\ord_x f_i(x,\xi(x))=m_i$,
 $\ord_x f_i(x,y)=n_i$.
 Consider the tropical polynomial
 $$\mu_f(t)=\sum_{i=1}^r \min(n_i t,m_i).$$
 Then  $\hat\mu(t)\ge \mu_f(t)/d,$ with equality at all values
 of $t$ such that $\mu_f(t)>d \sqrt{t}$.
\end{Pro}

\begin{proof}
 It is immediate that $\vv[f]{\xi}{t}=\mu_f(t)$, so the inequality
 $\hat\mu(t)\ge \mu_f(t)/d$ is clear. Now assume that $\mu_f(t)>d \sqrt{t}$.
 This implies that $\vv{\xi}{t}$ is not minimal, and therefore by
 Lemma \ref{negcurves}, $f$ computes
 $\hat\mu (\vv{\xi}{t})=\hat\mu ({t})$.
\end{proof}

\begin{Exa}
The easiest examples of the situation described in Proposition
\ref{genericneg} are given by (smooth) curves of degree 1 and 2.

Namely, for $d=1$, $m_1/n_1=2$,  it is trivial that for general
$\xi(x)$, there exists a degree 1 polynomial $f$ with
 $\ord_x f(x,\xi(x))=2$, $\ord f_i(x,y)=1$; one simply has to
 take the equation of the tangent line to $y-\xi(x)=0$,
 or $f=y-\xi_1(x)$ (where $\xi_1$ denotes the 1-jet).

In the same vein,  for $d=2$, $m_1/n_1=5$, it is easy to show that for general
$\xi(x)$, there exists a degree 2 polynomial $f$ with
 $\ord_x f(x,\xi(x))=5$, $\ord f_i(x,y)=1$, which for general $\xi$
 is irreducible; one simply has to
 take the equation of the conic through the first five points infinitely near
 to $(0,0)$ on the curve $y-\xi(x)=0$ (more fancily, the curvilinear
 ideal $(y-\xi(x))+(x,y)^5\subset \C[x,y]$ has maximal Hilbert function and
 colength 5, and therefore a unique element in degree 2
 up to a constant factor).

Proposition \ref{genericneg} then gives that
$$
\hat\mu(t)=
\begin{cases}
 t & \text{ if }1\le t\le 2, \text{ computed by a line},\\
 2 & \text{ if }2\le t\le 4, \text{ computed by a line},\\
 t/2 & \text{ if }4\le t\le 5, \text{ computed by a conic},\\
 5/2 & \text{ if }5\le t\le 25/4, \text{ computed by a conic}.
\end{cases}
$$
\end{Exa}

In order to construct the supraminimal curves in general position
computing the function $\hat\mu$ for small values of $t$,
we need certain Cremona maps, 
presumably well known,
which have been used by Orevkov in \cite{Ore02} to show sharpness of his
bound on the degree of cuspidal rational curves.

\begin{Pro}\label{cremona8}
Let $K=(p_1,\dots,p_7)$ be a general cluster with $p_{i+1}$ infinitely
near to $p_{i}$ for $i=1,\dots,6$.
There exists a degree 8 plane Cremona map $\Phi_8$  whose cluster of
fundamental points is $K$, with all points weighted with multiplicity 3, and satifying  the
following properties:
\begin{enumerate}
 \item The characteristic matrix of $\Phi_8$ is
 $$\left( \begin {array}{cccccccc} 8&3&3&3&3&3&3&3\\ \noalign{\medskip}-
3&-1&-2&-1&-1&-1&-1&-1\\ \noalign{\medskip}-3&-2&-1&-1&-1&-1&-1&-1
\\ \noalign{\medskip}-3&-1&-1&-1&-2&-1&-1&-1\\ \noalign{\medskip}-3&-1
&-1&-2&-1&-1&-1&-1\\ \noalign{\medskip}-3&-1&-1&-1&-1&-1&-2&-1
\\ \noalign{\medskip}-3&-1&-1&-1&-1&-2&-1&-1\\ \noalign{\medskip}-3&-1
&-1&-1&-1&-1&-1&-2\end {array} \right). $$
\item The inverse Cremona map is of  the same type, i.e., it has the same
characteristic matrix and its fundamental points are a sequence, each
infinitely near to the preceding one.
\item The only curve contracted  by $\Phi_8$ is the nodal cubic which
is singular at $p_1$ and goes through
$(p_2,\dots,p_7)$. The only expansive fundamental point is $p_7$,
whose relative principal curve is the nodal cubic going through the
fundamental points of the inverse map, and singular at the first of them.
\end{enumerate}
\end{Pro}

Recall that the characteristic matrix of a plane Cremona map
is the matrix of base change in the Picard group of the blow
up $\pi:S\rightarrow \P^2$ that resolves the map, from the natural base formed by the
class of a line and the exceptional divisors, to the natural base
in the \emph{image} $\hat \P^2$, formed by the class of a line there
(the homaloidal net in the original $\P^2$) and the
divisors contracted by the map (which are the exceptional divisors
of $\pi':S\rightarrow \hat \P^2$), see \cite{AC02}. We use it later on to compute images
of curves under $\Phi_8$.

\begin{proof}
This proof is taken from \cite[p. 667]{Ore02}; the only modification lies in
the remark that $K$ can be taken general. Indeed,
for $K$  general, there exists a unique irreducible nodal cubic $\Gamma$
with multiplicity 2 at $p_1$ and going through $p_2, \dots, p_7$.
$\Phi_8$ is then defined as follows: let $\pi_K:S_K \rightarrow \P^2$
be the blowup of all points on $K$. The (proper) exceptional divisors
$\widetilde E_1, \dots, \widetilde E_6$ are $(-2)$-curves,
$E_7$ is a $(-1)$ curve. The proper transform
$\widetilde \Gamma\subset S_K$ is another $(-1)$-curve that meets
the (proper) exceptional divisors $\widetilde E_1$ and $E_7$. Blow down
$\widetilde \Gamma$, $\widetilde E_1, \dots \widetilde E_6$ to obtain another map
$\pi_K':S_K\rightarrow \hat \P^2$. Then take $\Phi_8=\pi_K'\circ \pi_K^{-1}$.
All the stated properties are easy to check.
\end{proof}

Denote $F_{-1}=1$, $F_0=0$ and $F_{i+1}=F_i+F_{i-1}$ the Fibonacci numbers,
and $\phi=(1+\sqrt{5})/2=\lim F_{i+1}/F_i$ the ``golden ratio''.

\begin{Pro}
 For each odd $i\ge 1$, there exist rational curves $C_i$ of degree $F_{i}$
 with a single cuspidal singularity of characteristic exponent $F_{i+2}/F_{i-2}$
 whose six singular free points are in general position. These curves become
 $(-1)$-curves in their embedded resolution, and are supraminimal for $t$
 in the interval $\left(\frac{F_{i}^2}{F_{i-2}^2},\frac{F_{i+2}^2}{F_{i}^2}\right)$.
\end{Pro}
Note that for $i=1$ the line is actually not singular (the ``characteristic exponent''
is 2, an integer) but the statement in that case means that the line goes through
the first two of six infinitely  near points in general position, i.e.,
the exponent is interpreted as $m_i/n_i=2$ in Proposition \ref{genericneg}.

\begin{proof}
 The existence of such curves, without the generality statement, is
 \cite[Theorem C, (a) and (b)]{Ore02}. Since the construction
 goes by recursively applying the rational map $\Phi_8$, and the
 free singular points of $C_i$ are exactly the seven fundamental
 points of $\Phi_8$, it follows from \ref{cremona8} that these can
 be chosen to be general. They are $(-1)$-curves after resolution
 because the starting point of the construction are the two lines tangent
 to the two branches of the nodal cubic $\Gamma$
 (which becomes an exceptional divisor after $\Phi_8$) i.e.,
 $(-1)$-curves (each is  a line through a point and an infinitely near point).

 Now, with notation as in Proposition \ref{genericneg},
 $$\mu_f(t)=
\begin{cases}
 \frac{F_{i-2}}{F_{i}}t & \text{if }t \le \frac{F_{i+2}}{F_{i-2}},\\
 \frac{F_{i+2}}{F_{i}} & \text{if }t \ge \frac{F_{i+2}}{F_{i-2}}.
\end{cases}
$$
supraminimality in the claimed interval follows.
\end{proof}

\begin{Cor} For every odd $i$,
 $$\hat\mu(t)=
\begin{cases}
 \frac{F_{i-2}}{F_{i}}t & \text{if }t \in \left[\frac{F_{i}^2}{F_{i-2}^2},\frac{F_{i+2}}{F_{i-2}}\right],\\
 \frac{F_{i+2}}{F_{i}} & \text{if }t \in \left[\frac{F_{i+2}}{F_{i-2}},\frac{F_{i+2}^2}{F_{i}^2}\right].
\end{cases}
$$

\end{Cor}

\begin{Rem}
We proved that supraminimal values of $\hat\mu$ are computed
by a single irreducible curve. In contrast, we see that the minimal
values at $t=F_{i+2}^2/F_{i}^2$ are computed both by $C_i$ and $C_{i+2}$.
In fact, the two divisors $F_{i+2}C_i$ and $F_iC_{i+2}$ generate a pencil whose
general members also compute $\hat\mu(t)$; they are
again unicuspidal rational curves classified in
\cite[Theorem 1.1, (c)]{FLMN07}.
\end{Rem}

\begin{Rem}\label{QuadField}
In addition to the preceding family of curves,
nine additional $(-1)$-curves compute $\hat\mu(t)$ for some range of $t$
(see Table \ref{table}).
The existence of these curves is proved as follows.
$D_1$ and $D_2$ are well known.
The rest are obtained by applying the Cremona map $\Phi_8$
to already constructed curves (the names chosen indicate that
curve $X^*$ is built from curve $X$).
Recall that, because the intervals where a degree $d$ curve $C$
computes $\hat\mu$ are those where $\mu_f(t)\ge d\sqrt{t}$
(Proposition \ref{genericneg} and continuity of $\hat\mu$)
the endpoints correspond to values of $t$ where $\hat\mu$
is minimal. Note that all such endpoints given in Table
\ref{table} are squares in $\mathbb{Q}$ or in the quadratic
field to which they belong. This is due to the fact that
they satisfy $\mu_f(t)= d\sqrt{t}$, and $\mu_f(t)$ is
a piecewise affine linear function of $t$ with rational
coefficients.
\end{Rem}

\begin{table*}
\centering
\begin{tabular}{@{}cccc@{}}
\toprule
Name & $(d;v_i)$ & $m_i/n_i$ & $t$ \\ \midrule
$D_1$ & $(3 ;2,1^{\times 6})$ & 1,7 & $\left[\phi^4,\left(\frac{8}{3}\right)^2\right]_{_{}}$\\
$D_2^*$ & $(48 ;18^{\times 7},3,2^{\times 7})$ & 7, $\left( 7+\frac{1}{8}\right)^{\times 2}$, 8 &
  $\left[\left( {\frac {24+\sqrt{457}}{17}} \right) ^{2},\left( 24-\sqrt {455} \right) ^{2}\right]_{_{}}$\\
$C_1^{**}$ & $(64 ;24^{\times 7},3^{\times 7},1^{\times 2})$ & $7^{\times 2}$, $7+\frac{1}{7+1/2}$, $7+\frac{1}{7}$ &
  $\left[\left( {\frac {32-\sqrt {177}}{7}} \right) ^{2},
  \left( {\frac {16+\sqrt {179}}{11}} \right) ^{2}\right]_{_{}}$\\
$D_1^*$ & $(24 ;9^{\times 7},2,1^{\times 6})$ & 7, $7+\frac{1}{7}$, 8 &
  $\left[\left( \frac{6+\sqrt{22}}{4}\right)^2,
  \left( 12-\sqrt {87} \right) ^{2}\right]_{_{}}$\\
$C_5^{*}$ & $(40 ;15^{\times 7},2^{\times 6},1^{\times 2})$ & $7^{\times 2}$, $7+\frac{1}{6+1/2}$ &
  $\left[ \left( {\frac {20+\sqrt {218}}{13}} \right) ^{2},
  \left(\frac{107}{40}\right)^2\right]_{_{}}$\\
$C_3^*$ & $(16 ;6^{\times 7},1^{\times 5})$ & 7, $7+\frac{1}{5}$ &
  $\left[ \left( \frac{8+\sqrt {29}}{5} \right) ^{2},
  \left(\frac{43}{16}\right)^2\right]_{_{}}$\\
$D_3$ & $(35 ;13^{\times 7},4,3^{\times 3})$ & $7+\frac{1}{4}$, 8&
  $\left[\left(\frac{35}{13}\right)^2,
  \left( {\frac {35-\sqrt {877}}{2}} \right) ^{2}\right]_{_{}}$\\
$C_1^*$ & $(8 ;3^{\times 7},1^{\times 2})$ & 7, $7+\frac{1}{2}$ &
  $\left[\left(\frac{4+\sqrt{2}}{2}\right)^2,\left(\frac{22}{8}\right)^2\right]_{_{}}$\\
$D_2$ & $(6 ;3,2^{\times 7})$ & 1, 8 &
  $\left[\left( \frac{3+\sqrt {7}}{2} \right) ^{2},
  \left(\frac{17}{6}\right)^2\right]_{_{}}$\\
\bottomrule
\end{tabular}
\caption{Sporadic supraminimal curves. Here $(d;v_i)$ denotes
the degree and multiplicities sequence, with $^{\times k}$ meaning
$k$-tuple repetition, and $m_i/n_i$ follows the notation of Proposition
\ref{genericneg}, with $^{\times 2}$ again meaning
repetition.
\label{table}}
\end{table*}

\begin{Exa}
As an example, let us show the existence of $D_1^*$.
Let $K=(p_1, \dots, p_8)$ be a general cluster with each point infinitely
near to the preceding one; we want to show that there is an irreducible
curve of degree 24 with three branches, two smooth, one of which
goes through $(p_1, \dots, p_7)$ and the other through all of $K$,
and one singular, with characteristic exponent 50/7.
Because $K$ is general, there exist a cubic curve $D_1$
with multiplicities $[2,1^6,0]$ on $K$ and another cubic $\Gamma$
through $K$ that has
a node at some other point $q_1$. Choose one of the branches of $\Gamma$
and let $q_2, \dots, q_7$ be the points infinitely near to $q_1$ on that
branch. Apply the Cremona map $\Phi_8$ based on $(q_1,\dots,q_7)$:
then $D_1^*=\Phi_8(D_1)$.
\end{Exa}

All these computations together show that indeed, $(-1)$-curves compute
$\hat\mu$ in the anticanonical range:

\begin{Teo}\label{lastthm}
 For $t\in A$, $\hat\mu(t)$ is computed by $(-1)$-curves; more precisely,
 the (infinitely many) curves $C_i$ and 7 of the curves in table \ref{table}.
\end{Teo}

Figure \ref{graph} shows $\hat\mu(t)$ in the ranges where it is known,
together with the lower bound $\sqrt{t}$.

The two curves $C_1^{**}$ and $C_5^*$ compute $\hat\mu$ in ranges
of $t$ which do not intersect the anticanonical locus $A$. We expect that
there are no more curves with such behavior, and so propose the following
strengthening of conjecture \ref{nqmvconj}:

\begin{Con}\label{strong}
 Let $t\in \R$ be such that $\hat\mu(t)>\sqrt t$. Then $\mu_C(t)>\sqrt t$
 for a curve $C$ which is either on the list of table \ref{table} or one of the $C_i$.
 Equivalently, if $t>7+1/9$ is not contained in any one of the
 intervals of table \ref{table}, then a very general valuation
 $\vv{\xi}{t}$ is minimal.
\end{Con}

Obviously, Conjecture \ref{strong} implies Conjecture \ref{mainconj}.

\begin{Rem}
 For $t>(17/6)^2$, it is possible to show (using Cremona maps)
 that no $(-1)$-curve is ever supraminimal. Thus conjecture \ref{strong}
 splits naturally into two conjectures: first, that all supraminimal curves
 are $(-1)$-curves, and second, that the only supraminimal $(-1)$-curves
 in the interval $[7,8]$ are the ones above. Our evidence for the latter
 statement is experimental, obtained by a computer search.
\end{Rem}

\section*{Acknowledgements}
This research was supported through the programme ``Research in Pairs''
by the Mathematisches Forschungsinstitut Oberwolfach in 2013.
Alex K\"uronya was partially supported  by the DFG-Forschergruppe 790
''Classification of Algebraic Surfaces and Compact Complex Manifolds``,  by
the DFG-Graduiertenkolleg 1821 ''Cohomological Methods in Geometry``,  and by
the OTKA grants 77476 and  81203 of the Hungarian Academy of Sciences.
Brian Harbourne's work on this project
was partially sponsored by the National Security Agency under Grant/Cooperative
agreement ``Advances on Fat Points and Symbolic Powers,'' Number H98230-11-1-0139.
The United States Government is authorized to reproduce and distribute reprints
notwithstanding any copyright notice.
Joaquim Ro\'e was partially supported by MINECO grant
MTM2013-40680-P.
Marcin Dumnicki and Tomasz Szemberg were partially supported
by the National Science Centre, Poland, grant 2014/15/B/ST1/02197.
We thank the referees for valuable comments and remarks.

\bibliographystyle{amsplain}

\end{document}